\documentclass{article}

\usepackage{CJK}
\usepackage{CJK,CJKnumb,CJKulem,times,dsfont,ifthen,mathrsfs,latexsym,amsfonts,color}
\usepackage{amsmath,amsthm,makeidx,fontenc,amssymb,bm,graphicx,psfrag,listings,curves,extarrows}

\usepackage[utf8]{inputenc}
\usepackage[english]{babel}
\usepackage{multicol,graphicx,color}
\usepackage{pslatex}
\usepackage{authblk}
\usepackage{amsthm}
\usepackage{cite}
\usepackage{amsmath}
\usepackage{amssymb}
\usepackage{latexsym}
\usepackage{lscape}
\usepackage{epsfig}
\usepackage{amsfonts}
\usepackage{mathrsfs,bbm}
\usepackage[
hmarginratio={1:1},     
vmarginratio={1:1},     
textwidth=15cm,        
textheight=21cm,
heightrounded,          
]{geometry}

\usepackage{tikz}

\usepackage[
linktocpage=true,colorlinks,citecolor=magenta,linkcolor=blue,urlcolor=magenta]{hyperref}
\makeatletter

\usepackage[showonlyrefs]{mathtools}  
\mathtoolsset{showonlyrefs=true}

\newtheorem{theorem}{Theorem}[section]

\newtheorem{lemma}[theorem]{Lemma}
\newtheorem{proposition}[theorem]{Proposition}
\newtheorem{remark}[theorem]{Remark}

\theoremstyle{definition} \theoremstyle{remark}
\numberwithin{equation}{section}

\newcommand{\R}{\mathbb{R}}
\newcommand{\la}{\lambda}
\newcommand{\ep}{\epsilon}
\newcommand{\G}{\mathcal{G}}
\newcommand{\p}{\mathcal{P}}
\newcommand{\cR}{\mathcal{R}}

\begin{document}
\title{{\bf Uniqueness of positive solutions of double-power nonlinear stationary Schr\"odinger equations:} 	
\\the fixed frequency case and the fixed mass case
}

	\author{{\bf Linjie Song}\footnote{songlinjie18@mails.ucas.edu.cn} }

	\author{{\bf Wenming Zou}\footnote{zou-wm@mail.tsinghua.edu.cn} }
	
	\affil{{\small Department of Mathematical Sciences, Tsinghua University, Beijing 100084, China}}
	

	\date{}	
	\maketitle
	
	\begin{abstract}
\noindent The existence of a positive radial solution to the problem
		\begin{align*}
			-\Delta u + \la u = u^p + u^q, \quad u \in H^1(\R^N), ~~ N \ge 2,
		\end{align*}
		where  $\la > 0$ and $1 < q < p < 2^* - 1$,  has been known for a long time. For the pure-power nonlinearity, it is well known that this solution is unique. In contrast, the uniqueness problem for the double-power case is substantially more delicate and depends on the dimension and the exponents. It has been known that the uniqueness result is in general not true for the equation with double-power nonlinearities in dimension three and it is a long-standing open question to prove the uniqueness of positive radial solutions throughout the full subcritical range for $N \ge 4$ since the work of H. Berestycki, P.-L. Lions \cite{BL} (Arch. Ration. Mech. Anal. 1983). In this paper we provide the uniqueness results for all $\la > 0$ and all $1 < q < p < 2^* - 1$ when $N \ge 5$. In dimensions $N = 2,3,4$, with additional conditions on $p$ and $q$, we obtain the uniqueness results for all $\la > 0$. Our results are in sharp contrast to those of  J. D\'avila, M. del Pino and I. Guerra \cite{DPG}  (Proc. London Math. Soc.  2013), where  some non-uniqueness results were obtained  for $N=3$. At the same time, we give a positive answer and rigorous  proof to the implication of their numerical simulation,   clarify  the open issues left by \cite{DPG}. Finally, we completely resolve the uniqueness problem of the above equation with $L^2$-mass constraints in the mass-subcritical and mass-critical regimes for all $N \ge 2$.

		\vskip0.1in
		{\small \noindent \text{\bf Keywords:} Uniqueness; Positive solutions; Morse index; Nondegeneracy; Double-power nonlinearity; Nonlinear Schr\"odinger equations
		\vskip0.1in
		\noindent\text{\bf Mathematics Subject Classification:} 35A02, 35J61, 35B09
		\vskip0.1in
		\noindent \textbf{Statements and Declarations:} The authors have no relevant financial or non-financial interests to disclose.
		\vskip0.1in
		\noindent \textbf{Data availability:} Data sharing is not applicable to this article as no datasets were generated or analysed during the current study.
		\vskip0.1in
		\small \noindent \textbf{Acknowledgements:} This work is funded by National Key R\&D Program of China (Grant 2023YFA1010001), NSFC(12571123) and NSFC(12671139). It was completed when Song was visiting the Universit\'e Marie et Louis Pasteur, LmB (UMR 6623). Song is supported by the China Postdoctoral Science Foundation (2024T170452).
		}
		
	\end{abstract}
	
	\newpage
	\tableofcontents

	\bigskip
	
	\section{Introduction and main results} \label{secintroduction}
	In this paper we consider the following  nonlinear stationary Schr\"odinger equation with two
 power-type nonlinearities:
\begin{align} \label{dpequations}
		-\Delta u + \la u = |u|^{p-1}u + |u|^{q-1}u \quad \text{in } \R^N,
	\end{align}
	where $N \ge 2$, $\la > 0$ and the powers $p$ and $q$ are superlinear and subcritical, namely
	\begin{align*}
		1 < q < p < 2^* - 1; \quad\quad 2^* = \frac{2N}{N-2} ~ \text{ if } N \ge 3, \quad 2^* = \infty ~ \text{ if } N = 2.
	\end{align*}
Solutions of \eqref{dpequations} correspond to standing wave solutions $\psi(t,x) = e^{-i\la t}u(x)$ to the following time-dependent Schr\"odinger equation:
\begin{align}\label{doubletimedependeq}
i\psi_t =\Delta \psi  +  |\psi|^{p-1}\psi + |\psi|^{q-1}\psi    \quad \text{ in } \R \times \R^N.
\end{align}
This equation is a natural non-scaling invariant extension of the extensively studied equation with a pure-power nonlinearity:
\begin{align}\label{zwm=1}
i\psi_t = \Delta \psi  +  |\psi|^{p-1}\psi    \quad \text{ in } \R \times \R^N.
\end{align}
The  theories on the basic problems  such as the  well-posedness, asymptotic behaviour and blow-up phenomenon
about \eqref{doubletimedependeq} and \eqref{zwm=1} were  considered  by  T. Tao, M. Visan and X. Zhang \cite{tao} (see also  J. D\'avila, M. del Pino and I. Guerra \cite{DPG}) and the references therein. However, the uniqueness of positive solutions has always been a concern of physicists and mathematicians, which is very challenging.

\vskip0.12in

We are interested in the uniqueness of positive decaying solutions  to the equation \eqref{dpequations}, namely the positive solutions of the problem:
	\begin{align} \label{doublepower}
		\left\{
			\begin{aligned}
				& -\Delta u + \la u = u^{p} + u^{q} \quad \text{in } \R^N, \\
				& u > 0, \quad u(x) \to 0 ~~\text{ as }~ |x| \to \infty.
			\end{aligned}
		\right.
	\end{align}
The existence of the  solution of \eqref{doublepower} can be established by using the mountain-pass lemma or by searching for a minimizer on the Nehari manifold, see for example \cite{Wi}. Compared to the existence problem, the issue of the uniqueness is much more challenging, and people know less about it. In fact, this question  is a major open problem proposed in \cite[Section 10]{BL}, in which Berestycki and Lions asked for which nonlinearities $f$ the following semilinear elliptic equation admits a unique positive radial solution:
\begin{align*}
-\Delta u = f(u) \quad \text{ in } \R^N, \quad\quad u \in H^1(\R^N)
\end{align*}
with $N \ge 3$.

\vskip0.12in

In the case of a pure-power nonlinearity, namely the problem:
	\begin{align} \label{purepower}
		\left\{
			\begin{aligned}
				& -\Delta u + \la u = u^{p} \quad \text{in } \R^N, \\
				& u > 0, \quad u(x) \to 0 ~~\text{ as }~ |x| \to \infty,
			\end{aligned}
		\right.
	\end{align}
	solutions of \eqref{purepower} (and also those of \eqref{doublepower}) are necessarily radially symmetric up to translations owing to the classical Gidas, Ni and Nirenberg result \cite{GNN}, and later in \cite{Kwong}, Kwong established uniqueness of the radially symmetric solution of \eqref{purepower} with $\la = 1$ for all $p \in (1,2^*-1)$. Since \eqref{purepower} is scaling invariant, Kwong's uniqueness result on \eqref{purepower} still holds when $0 < \la \neq 1$.

\vskip0.12in

	After Kwong's result, there are some extensions for more general nonlinearities, see, for instance, \cite{CL,KZ,M,ST}. Except for \cite{CL}, nonlinearities $f(u) = u^p + u^q - \la u$ are not covered in other references unless $p = q$. In \cite{CL}, the uniqueness result of \eqref{doublepower} holds only when $1 < q < p \le N/(N-2)$ and $N \ge 3$. Note that $N/(N-2) < 2^*-1$. This exponent appears due to the technical limitations and it does not mean that \eqref{doublepower} admits a non-uniqueness result when $p > N/(N-2)$.

\vskip0.12in

Since the publication of the pioneering article\cite{BL},	a long-standing open question  is
\begin{itemize}
\item[\large (\textbf{Q})] Does \eqref{doublepower} admit a unique solution up to translations for all $\la > 0$ and all $1 < q < p < 2^*-1$?
\end{itemize}
	In this direction, an important result was obtained  by D\'avila, del Pino and Guerra in \cite{DPG}. They considered the following equation in $\R^3$  with two power-type nonlinearities:
\begin{align}\label{dpg=1}
\left\{
\begin{aligned}
&\Delta v-v+v^p+\beta  v^q=0, \;\; v>0\;  \hbox{ in  } \; \R^3,\\
& v(x)\to 0\;  \hbox{ as } \; |x|\to \infty,
\end{aligned}
\right.
\end{align}
and proved that when $1<q<3,$ then for each $\beta $ sufficiently large, there exists a number $p_0\in (1, 5)$ such that problem \eqref{dpg=1} has at least {\bf three} solutions for
all $p\in (p_0, 5)$. On the other hand,  if $2\leq q < 3 $, then  problem \eqref{dpg=1} has  {\bf two} large solutions which can be explicitly described, see \cite[Theorems 1.1 and  1.3]{DPG}. Thus, Kwong's uniqueness result is in general not true for the equation with two power-type nonlinearities. Moreover, it also tells us that the answer to (\textbf{Q}) depends on the dimension $N$. Motivated by these phenomena, it is interesting to determine whether there exists some $ N \ge 2$ such that the answer to (\textbf{Q}) is "yes". In \cite{DPG}, according to numerical results, the authors pointed out that the "yes" should occur in dimensions $N \ge 4$ and that establishing the uniqueness of solution to \eqref{doublepower} appears as a challenging problem.

\vskip0.16in

	In this paper, we will provide an affirmative and complete answer to (\textbf{Q}) for all $N \ge 5$. Furthermore, when $N = 2,3,4$, we still obtain the uniqueness result of \eqref{doublepower} as long as  $p$  and  $q$  are in the proper range.  In particular, when $N=3$, our uniqueness results  are in sharp contrast to the multiplicity results of  J. D\'avila, M. del Pino and I. Guerra \cite{DPG}. Our main result in current paper  is as follows.

\newpage

	\begin{theorem}[Uniqueness: the fixed frequency case] \label{thmmain}
		Let $\la > 0$ and let $N, p, q$ satisfy one of the following assumptions:
		\begin{itemize}
			\item[(1)] $N \ge 5$, $1 < q < p < 2^* -1$;
			\item[(2)] $N = 2,3,4$, $1 < q < p \le p_{\text{alg}}(N) := \frac{4 + 2\sqrt{2N + 4}}{N}$;
			\item[(3)] $N = 2,3,4$, $1 < q < p < 2^* -1$, $p - q \le 1$.
		\end{itemize}
		Then \eqref{doublepower} admits a unique solution up to translations.
	\end{theorem}

	\begin{remark}
		When $N =2$, we think that the answer to (\textbf{Q}) is "no", namely, in suitable ranges for the parameters $p, q$ and $\la$, problem \eqref{doublepower} has at least two solutions. We are preparing a paper to prove it rigorously.	Once it is proved, in order to obtain uniqueness result of \eqref{doublepower} in dimension $2$, additional assumptions on $p$ and $q$ are necessary.	
	\end{remark}

		Let $u$ be a radial positive solution of \eqref{doublepower} and let
		\begin{align*}
			L_u := -\Delta + \la - pu^{p-1} - qu^{q-1}
		\end{align*}
		be the linearized operator defined on $L^2(\R^N)$ with domain $H^2(\R^N)$ and form domain $H^1(\R^N)$. The Morse index of $u$, denoted by $m(u)$, is defined as the sum of the dimensions of the negative eigenspace of $L_u$. Similarly, the radial Morse index $m_{rad}(u)$ is defined as the sum of the dimensions of the negative radial eigenspace of $L_u$. We say $u$ is nondegenerate in $H^1_{rad}(\R^N)$ if
		\begin{align*}
			\ker L_u|_{L^2_{rad}(\R^N)} = \{0\},
		\end{align*}
		and is nondegenerate in $H^1(\R^N)$ if
		\begin{align*}
			\ker L_u = \text{span}\big\{\partial_1 u, \cdots, \partial_N u\big\}.
		\end{align*}
		It is not difficult to see that the solution of \eqref{doublepower} obtained as the minimizer on Nehari manifold has Morse index $1$. In previous studies, people usually establish  firstly the uniqueness result, which yields that any solution is exactly the minimizer on Nehari manifold and thus has Morse index $1$, and then investigate its nondegeneracy. In the current paper, we adopt a different approach. More precisely,  without establishing the uniqueness result, we deal with firstly  two questions:  one is whether any solution of \eqref{doublepower} has Morse index $1$, the other is whether the radial solutions of \eqref{doublepower}  with Morse index $1$ are nondegenerate or not. Using the nondegeneracy result, we can establish a local branch of radial solutions parameterized by $\la$ and extend it globally. Then, by the local uniqueness result as $\la \to \infty$, we derive that the global branch is unique. The results of Morse index and nondegeneracy read as follows.

		\begin{theorem}[Morse index] \label{thmMorse}
			Under the assumptions of Theorem \ref{thmmain}, for any solution $u$ of \eqref{doublepower}, we have $m(u) = 1$.
		\end{theorem}

		\begin{theorem}[Nondegeneracy] \label{thmnondegenerate}
			Let $\la > 0$, $N \ge 2$ and $1 < q < p < 2^*-1$. Assume that either one of the conditions in Theorem \ref{thmmain} holds or
			\begin{align*}
				N(p-1) \le 2(q+1).
			\end{align*}
		Then every radial solution of \eqref{doublepower} having Morse index $1$ is nondegenerate in $H^1_{rad}(\R^N)$ and nondegenerate in $H^1(\R^N)$.
		\end{theorem}

		\medbreak
		Next, we turn our attention to the uniqueness of positive solutions with a prescribed mass. This result has many applications in other problems, such as proving the orbital stability of standing wave solutions, see Remark \ref{rmkstability} below. Let us consider the following equation:
		\begin{align} \label{eqmass}
		\left\{
			\begin{aligned}
				& -\Delta u + \la u = u^{p} + u^{q} \quad \text{in } \R^N, \\
				& u > 0, \quad u(x) \to 0 ~~\text{ as }~ |x| \to \infty, \\
				& \int_{\R^N}u^2dx = c,
			\end{aligned}
		\right.
	\end{align}
	where $N \ge 2$, $1 < q < p \le 1 + 4/N$, $c > 0$ is prescribed and $\la \in \R$ is unknown. Let $Q \in H^1_{rad}(\R^N)$ be the unique radial positive solution of
	\begin{align*}
		-\Delta Q + Q = Q^{1 + 4/N} \quad \text{ in } \R^N
	\end{align*}
	and $c_0 := \int_{\R^N}Q^2dx$. For any $c > 0$ when $1 < q < p < 1 + 4/N$ and for any $c \in (0,c_0)$ when $1 < q < p = 1 + 4/N$, the existence of a solution of \eqref{eqmass} can be established by finding a minimizer on the mass constrained sphere
	\begin{align*}
		S_c := \big\{u \in H^1(\R^N): \int_{\R^N}u^2dx = c \big\},
	\end{align*}
	and $\la$ appears as a Lagrange multiplier. As for the uniqueness issue, even if uniqueness of \eqref{doublepower} is established for any $\la > 0$, two different $\la_1$ and $\la_2$ may correspond to two solutions $u_1$ and $u_2$ having the same $L^2(\R^N)$-norm. In this sense, the uniqueness problem in the fixed mass case is more challenging than the fixed frequency case. Indeed, there are very few results in this regard. In \cite{HS2}, the first author and his co-author obtained some uniqueness results with an additional condition ($A_2$), but ($A_2$) was only proved when $N=1$ and remained open for higher dimensions, see Theorem 3.5 in \cite{HS2}. In this paper, we completely resolve the uniqueness problem of \eqref{eqmass} in the stated exponent range for all $N \ge 2$.

	\begin{theorem}[Uniqueness: the fixed mass case] \label{thmmain2}
		Let $N \ge 2$.
		\begin{itemize}
			\item If $1 < q < p < 1 + 4/N$, for any $c > 0$, \eqref{eqmass} admits a unique solution $(u_c,\la_c) \in H^1_{rad}(\R^N) \times (0,\infty)$.
			\item If $1 < q < p = 1 + 4/N$, for any $c \in (0,c_0)$, \eqref{eqmass} admits a unique solution $(u_c,\la_c) \in H^1_{rad}(\R^N) \times (0,\infty)$; for any $c \ge c_0$, \eqref{eqmass} admits no solution.
		\end{itemize}	
	\end{theorem}

	\begin{remark} \label{rmkstability}
		Define the energy functional
		\begin{align}
			I(u) := \frac{1}{2}\int_{\R^N}|\nabla u|^2dx - \frac{1}{p+1}\int_{\R^N}|u|^{p+1}dx - \frac{1}{q+1}\int_{\R^N}|u|^{q+1}dx.
		\end{align}
		When $1 < q < p < 1+4/N$ and $c > 0$, Lions showed in \cite{Lions} that the set
		\begin{align*}
			M := \big\{u \in S_c: I(u) = e_0\big\}, \quad e_0 := \inf_{S_c}I
		\end{align*}
		is not empty. The same result holds when $1 < q < p = 1+4/N$ and $c \in (0,c_0)$. It is clear that any element in $M$ is a positive (multiplying $-1$ if it is negative) solution of \eqref{eqmass}. Let
		\begin{align}
			M_0 := \big\{u \in H^1(\R^N;\mathbb{C}): \int_{\R^N}|u|^2dx = c, I(u) = e_0\big\}.
		\end{align}
		By classical arguments developed by Cazenave and Lions in \cite{CLions}, the set $M_0$ is orbitally stable. As an application of Theorem \ref{thmmain2}, we get
		\begin{align*}
			M_0 = \big\{e^{i\theta} u_c(\cdot + y): \theta \in \R, y \in \R^N \big\},
		\end{align*}
		and any standing wave solution corresponding to an element in $M_0$ is orbitally stable, where $u_c$ is the unique radial positive solution given in Theorem \ref{thmmain2}.
	\end{remark}

	In Section \ref{secpre1}, we perform a scaling transformation to convert Theorem \ref{thmMorse} and Theorem \ref{thmnondegenerate} into an equivalent problem for the equation \eqref{eqofv}. Then we list standard decay and asymptotic results for positive solutions, and obtain some useful estimates \eqref{usefulestimate1} and \eqref{usefulestimate2}. In Section \ref{sectransformation}, when $N \ge 3$, we transform the estimate of Morse index into a spectral lower bound estimate problem, and further transform it into proving $\G > 0$ by choosing a suitable test function, where $\G$ is given in \eqref{def:G}. When $N \ge 6$, $1 < q < p < 2^*-1$, or $2 \le N \le 5$, $1 < q < p \le p_{\text{alg}}(N)$, or $2 \le N \le 5$, $1 < q < p < 2^*-1$, $p - q \le 1$, we prove $\G > 0$, see Proposition \ref{propG>01} and Proposition \ref{propG>02}. The more challenging case when $N = 5$, $1 < q < p < 7/3$ is addressed in Section \ref{G>0whenN=5}. One of the main ingredients is the introduction of $\cR$ given in Step 1 in the proof of Proposition \ref{propG>03}, which is exactly zero on the 5-dimensional Talenti bubble and is positive for the double power-type  equation having a lower order term. This gives the comparison of $L^{p+1}$-moment and $L^{q+1}$-moment with the 5-dimensional Talenti bubble $W_0$ given in \eqref{def:W_0(r)} and further yields the key estimate \eqref{C<}. Note that the 4-dimensional Talenti bubble is not in $L^2$, thus arguments developed in Section \ref{G>0whenN=5} do not work when $N = 4$. In Section \ref{secMorse}, we complete the proof of Theorem \ref{thmMorse}. The case $N = 2$ needs additional arguments since Lemma \ref{lemtransformation} fails to work in this case. In Section \ref{secnondegenerate}, we focus on the proof of Theorem \ref{thmnondegenerate}. Suppose by contradiction that solution $v$ of \eqref{eqofv} is degenerate in $H_{rad}^1(\R^N)$ and $w \in H_{rad}^1(\R^N)$ is the second radial eigenfunction of the linearized operator. We prove $w/v$ is strictly decreasing in $r > 0$. Then we use two independent methods to obtain a contradiction. In Subsection \ref{subsecspectralmethod}, we use the weighted spectral method to find a contradiction when $\G > 0$ holds. In Subsection \ref{subM-Pmethod} we arrive at the desired conclusion by comparing inequalities obtained by using moment identities and Pohozaev identity when $N \ge 2$, $1 < q < p < 2^* - 1$, $N(p-1) \le 2(q+1)$. Section \ref{secuniqueness1} proves Theorem \ref{thmmain} by combining nondegeneracy, continuation of positive radial branches, and uniqueness in the large-frequency limit. Section \ref{secuniqueness2} proves strict monotonicity of the mass along the resulting global branch when $N\ge 2$, $1 < q < p \le 1 + 4/N$, and completes the proof of Theorem \ref{thmmain2}.

\vskip0.3in

Before closing the section, we  will give a very brief note to introduce some articles and results on the uniqueness of positive solutions of elliptic partial differential equations, although it has no direct relationship with our current article. In \cite{Aftalion2004},  A. Aftalion and  F. Pacella studied  the positive radial solutions of the  $p$-Laplace equation defined on the  unit ball with  Dirichlet boundary condition. They showed
 the uniqueness and nondegeneracy of positive radial solutions.
In \cite{Ni+Nussbaum}  W.-M.  Ni  and R.D. Nussbaum studied the  uniqueness of positive solutions of the nonlinear Dirichlet problem
$\Delta u + f(u, r) = 0$ on  a ball or an annulus in $\R^N$.  In \cite{Cortazar}, C. Cort\'azar,  M. Elgueta and P. Felmer considered the uniqueness of positive solutions of $\Delta u+f(u)=0$  in $\R^N, N\geq 3$, where $ uf'(u)/f(u)$ is monotone decreasing in a subset of $[0, +\infty)$.
In \cite{Troy2016},  W.C. Troy obtained the uniqueness of positive ground state solutions of the logarithmic Schr\"odinger equation.
In the recent paper \cite{MTang=Inv} by M. Tang, the  uniqueness of bound states to
$-\Delta u +u=|u|^{p-1} u$ in $\R^n, n\geq 3$ was proved. Hence, he gave a positive answer to a conjecture of Berestycki and Lions in 1983 on the
uniqueness of sign-changing bound states to $-\Delta u + u= |u|^{p-1} u $  in $\R^n,  u \in H^1(\R^n)$, $1<p<(n+2)/(n-2). $  It was  shown in \cite{MTang=Inv}   that, for each integer $k \geq  1,$  the problem has a unique solution $u = u(|x|),$    up to translation and reflection, that has precisely $k$ zeros for $|x| > 0.$ More conclusions,
we refer to readers to the references cited in \cite{Aftalion2004,Cortazar, DPG, Ni+Nussbaum,MTang=Inv, Troy2016,Yanagida}. For the ODE case, see the recent paper \cite{Feltrin}.
For non-uniqueness of positive solutions for supercritical semilinear heat equations, see for example \cite{Hisa}.

	\section{Preliminaries for the proof of Theorem \ref{thmMorse} and Theorem \ref{thmnondegenerate}} \label{secpre1}

	Let $u$ be a solution of \eqref{doublepower} with $N \ge 2$, $\la > 0$ and $1 < q < p  < 2^*-1$. Using the standard elliptic regularity theory, one gets $u \in L^\infty(\R^N) \cap C^2(\R^N)$. By the iteration technique and using that $u > 0$, we can further see that $u \in C^\infty$. By the classical Gidas, Ni and Nirenberg result \cite{GNN}, we assume $u$ is radially symmetric up to translations. Moreover, let $u(r) = u(x)$ where $r = |x| \in [0,\infty)$ and we have $u'(r) < 0$ for $r > 0$.

	Let
	\begin{align} \label{scalingv}
		v(x) = u(0)^{-1}u\left(\frac{x}{\sqrt{u(0)^{p-1}+u(0)^{q-1}}}\right).
	\end{align}
	Then $v(0) = 1$, $0 < v \le 1$ and $v(x)$ satisfies the following equation:
	\begin{align} \label{eqofv}
		-\Delta v + \eta v = (1-\theta)v^p + \theta v^q \quad \text{ in } \R^N,
	\end{align}
	where
	$$
	\eta = \frac{\la}{u(0)^{p-1}+u(0)^{q-1}}, \quad \theta = \frac{u(0)^{q-1}}{u(0)^{p-1}+u(0)^{q-1}} \in (0,1).
	$$
	It is clear that $m(u) = m(v)$. Moreover,  the nondegeneracy of $u$ and the nondegeneracy of $v$ in $H^1(\R^N)$ (in $H_{rad}^1(\R^N)$) are equivalent.
The following two lemmas are standard decay and asymptotic estimates.

	\begin{lemma}
		For any $\kappa \in (0,\sqrt{\eta})$, there exist $C_\kappa, R_\kappa > 0$ such that
		\begin{align*}
			v(r) + |v_r(r)| + |v_{rr}(r)| \le C_\kappa e^{-\kappa r} \quad \text{ for } r \ge R_\kappa.
		\end{align*}
	\end{lemma}

	\begin{lemma} \label{lemlimitofk}
		Let
		\begin{align*}
			k(r) := -\frac{v_r(r)}{v(r)}.
		\end{align*}
		Then, as $r \to \infty$,
		\begin{align*}
			k(r) = \sqrt{\eta} + \frac{N-1}{2r} + O(r^{-2}), \quad k_r(r) = -\frac{N-1}{2r^2} + O(r^{-3}).
		\end{align*}
	\end{lemma}

	Let $\phi = -v_r > 0$, and let
	\begin{align*}
		L := -\Delta + \eta - (1-\theta)pv^{p-1} - \theta q v^{q-1}
	\end{align*}
	acting on $L^2(\R^N)$,
	\begin{align*}
		L_{rad} := -\frac{d^2}{dr^2} - \frac{N-1}{r}\frac{d}{dr} + \eta - (1-\theta)pv^{p-1} - \theta q v^{q-1}
	\end{align*}
	acting on $L^2(r^{N-1}dr)$. Differentiating \eqref{eqofv} with respect to $r$ we obtain
	\begin{align} \label{eqofphi}
		\left(L_{rad} + \frac{N-1}{r^2}\right)\phi = 0.
	\end{align}
 	We introduce $s = \ln r$ and
	\begin{align*}
		\rho(s) = r^{N-2}\phi(r)^2.
	\end{align*}
	Take $\xi = -\ln v$ and set
	\begin{align*}
		& F(\xi) = (1-\theta)e^{-(p-1)\xi} + \theta e^{-(q-1)\xi}, \quad S = F -\eta, \\
		& P(\xi) = -F_\xi, \quad\quad\quad\quad\quad\quad\quad J(\xi) = (1-\theta)\frac{p-1}{p+1}e^{-(p-1)\xi} + \theta \frac{q-1}{q+1}e^{-(q-1)\xi}, \\
		& m(\xi) = -\frac{P_\xi}{P}, \quad\quad\quad\quad\quad\quad\quad q-1 < m < p-1.
	\end{align*}

	Let
	\begin{align*}
		k = -\frac{v_r}{v}, \quad~~ h = \frac{k}{r}, \quad~~ y = rk, \quad~~ z = -\frac{rh_r}{h} = 1 - \frac{k_r}{h}, \quad~~ T = N - y - z.
	\end{align*}
	We have
	\begin{align}
		& d(s) := (\ln \rho)_s = \frac{r\rho_r}{\rho} = N - 2(y+z), \\
		& y_s = ry_r = y(2-z) = y^2 + y - \frac{r^2v_{rr}}{v}, \label{ys}\\
		& z_s = rz_r = r^2\frac{h_r}{h}\frac{k_r}{k} - \frac{r^2k_{rr}}{k} = r^2P - zT - y(2-z). \label{zs}
 	\end{align}

	\begin{lemma} \label{lemestimate}
		Suppose $N \ge 2$ and $1 < q < p < 2^* -1$. Then
		\begin{align} \label{usefulestimate1}
			y_r > 0, \quad h_r < 0, \quad J - h < k_r < J
		\end{align}
		for $r > 0$.
	\end{lemma}

	\begin{proof}
		Let
		\begin{align*}
			\p_0 = r^N\left(v_r^2 + 2\mathcal{F}(v)\right) + (N-2)r^{N-1}vv_r
		\end{align*}
		where
		\begin{align*}
			\mathcal{F}(t) = -\frac{\eta}{2}t^2 + \frac{1-\theta}{p+1}t^{p+1} + \frac{\theta}{q+1}t^{q+1}.
		\end{align*}
		From direct computations we derive that
		\begin{align*}
			\p_0'(r) = r^{N-1}v^2\left(-2\eta + (1-\theta)\frac{N+2-p(N-2)}{p+1}v^{p-1} + \theta\frac{N+2-q(N-2)}{q+1}v^{q-1}\right).
		\end{align*}
		Since $v(r)$ is decreasing in $r > 0$ and $q < p < 2^* -1$, the function
		\begin{align*}
			-2\eta + (1-\theta)\frac{N+2-p(N-2)}{p+1}v(r)^{p-1} + \theta\frac{N+2-q(N-2)}{q+1}v(r)^{q-1}
		\end{align*}
		is decreasing in $r > 0$ and has at most one zero point. Furthermore, using $\p_0(0) = 0$ and $\p_0(r) \to 0$ as $r \to \infty$, we obtain $\p_0(r) > 0$ for all $r > 0$. On the other hand,
		\begin{align*}
			\p_0 = r^{N-2}v^2\left(ry_r - r^2J\right).
		\end{align*}
		Therefore,
		\begin{align*}
			y_r > rJ > 0.
		\end{align*}

		Next we prove $h_r < 0$. Since $v(0) = 1$, we have
		\begin{align*}
			h(0) = - \lim_{r \to 0}\frac{v_r(r)}{r} = -v_{rr}(0) =: h_0 > 0.
		\end{align*}
		Using \eqref{eqofv} we get
		\begin{align*}
			-N v_{rr}(0) = (1-\theta)v(0)^{p} + \theta v(0)^{q} - \eta v(0) = 1 - \eta.
		\end{align*}
		Thus $h_0 = (1-\eta)/N$. By the regularity and radial symmetry of $v(x)$, we can write
		\begin{align*}
			h(r) = h_0 + \mu r^2 + o(r^2), \quad h_r(r) = 2\mu r + o(r).
		\end{align*}
		Using $\xi_r = rh$, we get
		\begin{align*}
			\xi(r) = \int_0^r sh(s)ds = \int_0^r s\left(h_0 + \mu s^2 + o(s^2)\right)ds = \frac{h_0}{2}r^2 + O(r^4).
		\end{align*}
		Moreover, $F_\xi(0) = -P(0)$. Thus
		\begin{align*}
			F(\xi) = F(0) + F_\xi(0)\xi + o(\xi) = 1 - P(0)\xi  + o(\xi) = 1 - \frac{h_0P(0)}{2}r^2 + o(r^2).
		\end{align*}
		By $v_r/v = -rh$, we have
		\begin{align*}
			\frac{v_{rr}}{v} = -h - rh_r + r^2h^2.
		\end{align*}
		On the other hand,
		\begin{align*}
			-v_{rr} - \frac{N-1}{r}v_r + \eta v = F(\xi)v.
		\end{align*}
		Therefore, we can derive that
		\begin{align} \label{eqofh_r}
			rh_r = r^2h^2 + F -\eta - Nh.
		\end{align}
		Note that
		\begin{align*}
			& rh_r = r(2\mu r + o(r)) = 2\mu r^2 + o(r^2), \\
			& r^2h^2 = r^2(h_0 + \mu r^2 + o(r^2))^2 = h_0^2 r^2 + o(r^2),
		\end{align*}
		and
		\begin{align*}
			F - \eta - Nh & = 1 - \frac{h_0P(0)}{2}r^2 + o(r^2) - \eta - N (h_0 + \mu r^2 + o(r^2)) \\
			& = 1 - \eta - Nh_0 - \frac{h_0P(0) + 2N\mu}{2}r^2 + o(r^2) \\
			& = - \frac{h_0P(0) + 2N\mu}{2}r^2 + o(r^2).
		\end{align*}
		Comparing the coefficient of $r^2$, one can see that
		\begin{align*}
			2\mu = h_0^2 - \frac{h_0P(0) + 2N\mu}{2},
		\end{align*}
		namely
		\begin{align*}
			\mu = \frac{h_0(2h_0 - P(0))}{2(N+2)}.
		\end{align*}
		Thus
		\begin{align} \label{eqextending}
			h(r) = h_0 + \frac{h_0(2h_0 - P(0))}{2(N+2)} r^2 + o(r^2), \quad h_r(r) = \frac{h_0(2h_0 - P(0))}{N+2} r + o(r).
		\end{align}

		We claim that $2h_0 \le P(0)$. If it is not in this case, when $r > 0$ is small enough we have $h'(r) > 0$. On the other hand, as $r \to \infty$,
		\begin{align*}
			0 < h(r) = -\frac{v_r(r)}{rv(r)} \le \frac{C}{r} \to 0.
		\end{align*}
		This yields the existence of $R > 0$ such that
		\begin{align*}
			h_r(r) > 0 ~~ (0 < r < R), \quad\quad h_r(R) = 0.
		\end{align*}
		We also have $h_{rr}(R) \le 0$. Now we introduce
		\begin{align*}
			Q(r) := 2h(r) - P(\xi(r)).
		\end{align*}
		Direct computations show that
		\begin{align*}
			Q_r = 2h_r - P_\xi \xi_r = 2h_r - rh P_\xi.
		\end{align*}
		Since $h > 0$ and $P_\xi < 0$, we get $Q_r > 0$ in $r \in (0,R)$ and so $Q(R) > Q(0) = 2h_0 - P(0) > 0$. Differentiating \eqref{eqofh_r} with respect to $r$, we deduce that
		\begin{align*}
			rh_{rr} + h_r = 2r^2hh_r + 2rh^2 + F_\xi \xi_r - Nh_r = 2r^2hh_r + 2rh^2 - rhP - Nh_r,
		\end{align*}
		namely
		\begin{align} \label{eqofhrr}
			h_{rr} = hQ + \left(2rh - \frac{N+1}{r}\right)h_r.
		\end{align}
		This shows $h_{rr}(R) = h(R)Q(R) > 0$, a contradiction! Thus the claim holds true.

		In the following we aim to prove $h_r \le 0$ for $r > 0$. Suppose by contradiction that the set
		\begin{align*}
			E = \big\{r > 0: h_r(r) > 0\big\}
		\end{align*}
		is not empty. Take a connected component $(\alpha,\beta) \subset E$. It is clear that $h_r(r) > 0$ for $\alpha < r < \beta$. Since $h > 0$ and $h(r) \to 0$ as $r \to \infty$, we have $\beta < \infty$. At the right endpoint $r = \beta$,
		\begin{align*}
			h_r(\beta) = 0, \quad\quad h_{rr}(\beta) \le 0.
		\end{align*}
		If $\alpha > 0$, at the left endpoint $r = \alpha$ we also have
		\begin{align*}
			h_r(\alpha) = 0, \quad\quad h_{rr}(\alpha) \ge 0.
		\end{align*}
		Using \eqref{eqofhrr} we get
		\begin{align*}
			h_{rr}(\alpha) = h(\alpha)Q(\alpha).
		\end{align*}
		Together with $h(\alpha) > 0$ we obtain $Q(\alpha) \ge 0$. Reasoning as before we have
		\begin{align*}
			Q(\beta) > Q(\alpha) \ge 0.
		\end{align*}
		Using \eqref{eqofhrr} again,
		\begin{align*}
			h_{rr}(\beta) = h(\beta)Q(\beta) > 0,
		\end{align*}
		contradicting $h_{rr}(\beta) \le 0$. If $\alpha = 0$, we have proved that $Q(0) = 2h_0 - P(0) \le 0$. If $Q(0) < 0$, by \eqref{eqextending} one gets $h_r(r) < 0$ near $r = 0$, implying that $(0,\beta) \not\subset E$. Therefore, $Q(0) = 0$ and it still holds that
		\begin{align*}
			Q(\beta) > Q(0) = 0
		\end{align*}
		and
		\begin{align*}
			h_{rr}(\beta) = h(\beta)Q(\beta) > 0
		\end{align*}
		a contradiction! Thus the set $E$ must be empty, namely $h_r \le 0$ in $r > 0$.

		Further, we will eliminate the possibility that $h_r$ has a zero point. Arguing by contradiction, we assume that $h_r(R) = 0$ for some $R > 0$. Together with $h_r \le 0$ we get $h_{rr}(R) = 0$. By \eqref{eqofhrr} and using $h(R) > 0$ we have $Q(R) = 0$. At $r = R$,
		\begin{align*}
			Q_r(R) = 2h_r(R) - Rh(R)P_\xi(\xi(R)) = - Rh(R)P_\xi(\xi(R)) > 0.
		\end{align*}
		Differentiating \eqref{eqofhrr} with respect to $r$, we deduce
		\begin{align*}
			h_{rrr} = h_rQ + hQ_r + A_r(r)h_r + A(r)h_{rr}
		\end{align*}
		where
		\begin{align*}
			A(r) = 2rh - \frac{N+1}{r}.
		\end{align*}
		At $r = R$,
		\begin{align*}
			h_{rrr}(R) = h(R)Q_r(R) > 0.
		\end{align*}
		Expanding $h_r$ at $r = R$ we have
		\begin{align*}
			h_r(R + \ep) = h_r(R) + h_{rr}(R)\ep + \frac12 h_{rrr}(R)\ep^2 + o(\ep^2) = \frac12 h_{rrr}(R)\ep^2 + o(\ep^2).
		\end{align*}
		For small $\ep > 0$ we derive that $h_r(R + \ep) > 0$, contradicting $h_r \le 0$. This contradiction yields that $h_r < 0$ for all $r > 0$.

		Finally, by
		\begin{align*}
			y_r = (rk)_r = r k_r + k = rk_r + rh > rJ
		\end{align*}
		we get $k_r > J - h$. It remains to prove $k_r < J$. Let
		\begin{align*}
			\p_* = r^N\left(v_r^2 + 2\mathcal{F}(v)\right) + (N-1)r^{N-1}vv_r.
		\end{align*}
	    Then we have
		\begin{align*}
			\p_* = r^Nv^2 \left(k_r - J\right), \quad \p_*'(r) = r^{N-1}v^2 \Sigma
		\end{align*}
		where
		\begin{align*}
			\Sigma(r) = k^2 - \eta + (1-\theta)\frac{N+1-p(N-1)}{p+1}v^{p-1} + \theta\frac{N+1-q(N-1)}{q+1}v^{q-1}.
		\end{align*}
		We first prove $\p_* \le 0$ in $r > 0$. Otherwise, since $\p_*(0) = 0$ and $\p_*(r) \to 0$ as $r \to \infty$, we can assume the existence of $R > 0$ such that
		\begin{align*}
			\p_*(R) > 0, \quad\quad \p_*'(R) = 0.
		\end{align*}
		Clearly, $\Sigma(R)  = 0$. Since
		\begin{align*}
			k_r = k^2 - \frac{v_{rr}}{v} = k^2 - \eta + (1-\theta)v^{p-1} + \theta v^{q-1} - (N-1)h,
		\end{align*}
		at $r = R$ we get
		\begin{align*}
			k_r(R) = N J(\xi(R)) - (N-1)h(R), \quad Rh_r(R) = k_r(R) - h(R) = N\left(J(\xi(R)) - h(R)\right).
		\end{align*}
		By $h_r < 0$ we get $h(R) > J(\xi(R))$. Thus
		\begin{align*}
			\p_*(R) = (N-1)R^{N}v(R)^2\left(J(\xi(R)) - h(R)\right) < 0,
		\end{align*}
		a contradiction!

		Now we suppose by contradiction that there exists $r_0 > 0$ such that $\p_*(r_0) = 0$. In view of $\p_* \le 0$, we know $\p_*'(r_0) = 0$ immediately. Arguing as before, we get $\p_*(r_0) < 0$ and find a contradiction. This shows $\p_* < 0$, namely $k_r < J$ in $r > 0$. The proof is complete.
	\end{proof}

	\begin{lemma} \label{lemsestimate}
		Suppose $N \ge 2$ and $1 < q < p < 2^* -1$. Let
		\begin{align} \label{def:beta}
			\beta = \frac{(N-1)(p-1)}{p+3} \in (0,1).
		\end{align}
		Then
		\begin{align} \label{usefulestimate2}
			k_r > J - \beta h \quad \text{ in } r > 0.
		\end{align}
	\end{lemma}

	\begin{proof}
		Let
		\begin{align*}
			\p_\beta = r^N v^2 (k_r + \beta h - J).
		\end{align*}
		By direct computations we derive that
		\begin{align*}
			\p_\beta'(r) = r^{N-1}v^2 \mathfrak{L}, \quad \mathfrak{L} = (1-\beta)k^2 + (1+\beta)S - NJ.
		\end{align*}
		Using $k_r < J$ we obtain
		\begin{align*}
			& ~~~ k^{-1}\mathfrak{L}_r  < 2(1-\beta)J + (1+\beta)S_\xi - N J_\xi \\
			& = (1-\theta)\frac{p-1}{p+1}\left((N-1)(p-1) - (p+3)\beta\right)v^{p-1} + \theta\frac{q-1}{q+1}\left((N-1)(q-1) - (q+3)\beta\right)v^{q-1} < 0,
		\end{align*}
		where we have used
		\begin{align*}
			\frac{(N-1)(q-1)}{q+3} < \frac{(N-1)(p-1)}{p+3} = \beta.
		\end{align*}
		Together with $\p_\beta(0) = 0$ and $\p_\beta(r) \to 0$ as $r \to \infty$, we get $\p_\beta > 0$, namely $k_r > J - \beta h$. The proof is complete.
	\end{proof}

	\section{Transformation to a one-dimensional spectral problem} \label{sectransformation}

	When $N \ge 3$, we transform the estimate of Morse index into a spectral lower bound estimate problem, and further transform it into proving $\G > 0$ by choosing a suitable test function, where $\G$ is given in \eqref{def:G}. When $N \ge 6$, $1 < q < p < 2^*-1$, or $2 \le N \le 5$, $1 < q < p \le p_{\text{alg}}(N)$, or $2 \le N \le 5$, $1 < q < p < 2^*-1$, $p - q \le 1$, we prove $\G > 0$ in Proposition \ref{propG>01} and Proposition \ref{propG>02}. In the coming section, we will address the more challenging case when $N = 5$, $1 < q < p < 7/3$.

	\begin{lemma} \label{lemtransformation}
		Suppose $N \ge 3$. For any radial function $f \in C_c^\infty(\R^N)$, let $f(r) = f(x)$ with $r = |x|$ and rewrite $f(r) = \phi(r)g(r)$. Then
		\begin{align} \label{indentity}
			Q_L(f) = \omega_{N-1} \int_{\R} \rho(s)\left(g_s^2 - (N-1)g^2\right)ds,
		\end{align}
		where $\omega_{N-1} = |\mathbb{S}^{N-1}|$ and
		\begin{align*}
			Q_L(f) := \int_{\R^N}\left(|\nabla f|^2 + \eta f^2 - (1-\theta)pv^{p-1}f^2 - \theta qv^{q-1}f^2\right)dx.
		\end{align*}
		The identity \eqref{indentity} can be extended to the case of $f \in H^1_{rad}(\R^N)$.
	\end{lemma}

	\begin{proof}
		Multiplying both sides of \eqref{eqofphi} by $r^{N-1}g^2\phi$ and integrating on $(\ep,R)$ yields that
		\begin{align*}
			\left[\phi\phi_r g^2r^{N-1}\right]_{\ep}^R - \int_\ep^R \left(2\phi\phi_r gg_r + \phi_r^2 g^2\right)r^{N-1}dr = \int_\ep^R\left(\frac{N-1}{r^2} + \eta - (1-\theta)pv^{p-1} - \theta q v^{q-1}\right)\phi^2 g^2 r^{N-1}dr.
		\end{align*}
		Then we have
		\begin{align} \label{eqeptoR}
			& \int_\ep^R \left(f_r^2 + (\eta - (1-\theta)pv^{p-1} - \theta q v^{q-1})f^2\right)r^{N-1}dr \\
			& ~~~ = \int_\ep^R \phi^2g_r^2 r^{N-1}dr - (N-1)\int_\ep^R \phi^2g^2r^{N-3}dr + \left[\phi\phi_r g^2r^{N-1}\right]_{\ep}^R.
 		\end{align}
		Recall that $h_0 = -v_{rr}(0) > 0$. Since $v(x) \in C^4(\R^N)$ is radial, we know
		\begin{align*}
			\phi(r) = h_0r + O(r^3), \quad \phi_r(r) = h_0 + O(r^2), \quad \rho(s) = h_0^2r^{N}(1+O(r^2)).
		\end{align*}
		Moreover, since $f$ is smooth and radial, we have
		\begin{align*}
			f(r) = f(0) + O(r^2), \quad f_r(r) = f_{rr}(0)r + O(r^3).
		\end{align*}
		Therefore,
		\begin{align*}
			g(r) = \frac{f}{\phi} = \frac{f(0)}{h_0r} + O(r), \quad g_r(r) = \frac{\phi f_r - f\phi_r}{\phi^2} = -\frac{f(0)}{h_0r^2} + O(1).
		\end{align*}
		Recalling $s = \ln r$, one gets $g_s = rg_r$, yielding that
		\begin{align*}
			g_s = r\frac{\phi f_r - f\phi_r}{\phi^2} = -\frac{f(0)}{h_0r} + O(r).
		\end{align*}
		Thus, near $r = 0$, both $g$ and $g_s$ belong to $L^2(\rho ds)$ for $N > 2$. If $N = 2$ and $f(0) \neq 0$, both integrals diverge. We also notice that
		\begin{align*}
			\phi\phi_r g^2r^{N-1} = f(0)^2r^{N-2} + O(r^N) \to 0 \quad \text{ as } r \to 0.
		\end{align*}
		Choose $R > 0$ large enough such that $\text{supp~}f(x) \subset B_R(0)$. Then sending $\ep$ to $0$ in \eqref{eqeptoR} and performing variable substitution $s = \ln r$, we derive that
		\begin{align*}
			Q_L(f) & = \int_{\mathbb{S}^{N-1}}d\sigma \int_0^\infty\left(f_r^2 + (\eta - (1-\theta)pv^{p-1} - \theta q v^{q-1})f^2\right)r^{N-1}dr \\
			& = \omega_{N-1} \int_{\R} \rho(s)\left(g_s^2 - (N-1)g^2\right)ds.
		\end{align*}
Using the radial Hardy inequality
		\begin{align*}
			\int_0^\infty f(r)^2r^{N-3}dr \le \frac{4}{(N-2)^2}\int_0^\infty f_r(r)^2r^{N-1}dr,
		\end{align*}
		we can extend the identity \eqref{indentity} to the case of $f \in H^1_{rad}(\R^N)$.
	\end{proof}

\vskip0.3in

	Now we define the weighted Poincar\'e constant
	\begin{align} \label{def:laPrho}
		\la_{\text{P}}(\rho) := \inf_{\substack{0 \neq g \in H^1(\rho ds) \\ \int_{\R}\rho gds = 0}} \frac{\int_{\R}\rho g_s^2ds}{\int_{\R}\rho g^2ds},
	\end{align}
	where
	\begin{align*}
		H^1(\rho ds) := \big\{g \in H^1_{loc}(\R): \int_\R \rho \left(g_s^2 + g^2\right)ds < \infty \big\}.
	\end{align*}
	Lemma \ref{lemtransformation} yields that when $N \ge 3$,
	\begin{align*}
		\la_{\text{P}}(\rho) \ge N-1 \Rightarrow m_{rad}(v) \le 1 .
	\end{align*}
	In the weighted space $L^2(\rho ds)$, we consider an operator defined by
	\begin{align} \label{def:Lrho}
		L_\rho g = -\frac{1}{\rho}\frac{d}{ds}(\rho g_s) = -g_{ss} - \frac{\rho_s}{\rho}g_s,
	\end{align}
	that is $L_\rho = - \frac{d^2}{ds^2} - d(s)\frac{d}{ds}$ where
	\begin{align*}
		d(s) = \frac{\rho_s(s)}{\rho(s)} = (\ln \rho(s))_s.
	\end{align*}
	For $g \in C_c^\infty(\R)$, direct computations yield that
	\begin{align*}
		\langle L_\rho g,g \rangle_{L^2(\rho ds)} = \int_{\R}g_s^2\rho(s)ds \ge 0.
	\end{align*}
	Thus $L_\rho$ is nonnegative and symmetric. Considering its Friedrichs extension, the form domain is $H^1(\rho ds)$. Define the unitary operator
	\begin{align*}
		U: L^2(\rho ds) \to L^2(\R), \quad Ug = \sqrt{\rho}g.
	\end{align*}
	Let
	\begin{align*}
		f = Ug = \sqrt{\rho}g, \quad g = \rho^{-1/2}f.
	\end{align*}
	Since
	\begin{align*}
		(\rho^{-1/2})_s = -\frac 12\frac{\rho_s}{\rho}\rho^{-1/2} = -\frac{d(s)}{2}\rho^{-1/2},
	\end{align*}
	we have
	\begin{align*}
		g_s = \rho^{-1/2} \left(f_s - \frac{d(s)}{2}f\right).
	\end{align*}
	This inspires us to introduce the first-order operator
	\begin{align*}
		A = \frac{d}{ds} - \frac{d(s)}{2}
	\end{align*}
	acting on $L^2(\R)$, with
	\begin{align*}
		D(A) = \big\{f \in L^2(\R): f \in H^1_{loc}(\R), f_s - \frac{d(s)}{2}f \in L^2(\R)\big\}.
	\end{align*}
	Note that $C_c^\infty(\R) \subset D(A)$ and $A$ is a closed operator.
	In the usual space $L^2(\R)$,
	\begin{align*}
		\left(\frac{d}{ds}\right)^* = -\frac{d}{ds}.
	\end{align*}
    Then we obtain
	\begin{align*}
		A^* = -\frac{d}{ds} - \frac{d(s)}{2},
	\end{align*}
	with
	\begin{align*}
		D(A^*) = \big\{f \in L^2(\R): f \in H^1_{loc}(\R), -f_s - \frac{d(s)}{2}f \in L^2(\R)\big\}.
	\end{align*}
	A formal calculus gives
	\begin{align*}
		A^* A f = \left(-\frac{d}{ds} - \frac{d(s)}{2}\right)\left(f_s - \frac{d(s)}{2}f\right) = -f_{ss} + \frac{d_s(s)}{2}f + \frac{d(s)^2}{4}f.
	\end{align*}
	We introduce the operator
	\begin{align*}
		H_0 = A^* A = -\frac{d^2}{ds^2} + \frac{d(s)^2}{4} + \frac{d_s(s)}{2},
	\end{align*}
	with
	\begin{align*}
		D(H_0) = \big\{f \in D(A): Af \in D(A^*)\big\}.
	\end{align*}
	Its closed quadratic form is
	\begin{align} \label{formdomainofH_0}
		\mathfrak{h}_0(f) = \|Af\|_{L^2(\R)}^2, \quad D(\mathfrak{h}_0) = D(A).
	\end{align}
	One can see that
	\begin{align*}
		U L_\rho U^{-1} = H_0.
	\end{align*}
	Thus $L_\rho$ and $H_0$ have the same spectrum.

	Define
	\begin{align*}
		\widetilde{H} = A A^* = -\frac{d^2}{ds^2} + \frac{d(s)^2}{4} - \frac{d_s(s)}{2},
	\end{align*}
	with
	\begin{align*}
		D(\widetilde{H}) = \big\{f \in D(A^*): A^*f \in D(A)\big\}.
	\end{align*}
	Its closed quadratic form is
	\begin{align} \label{formdomainoftildeH}
		\widetilde{\mathfrak{h}}(f) = \|A^*f\|_{L^2(\R)}^2, \quad D(\widetilde{\mathfrak{h}}) = D(A^*).
	\end{align}

	\begin{lemma}
		$\ker H_0 = \text{span}\{\sqrt{\rho}\}$ and $\ker \widetilde{H} = \{0\}$.
	\end{lemma}

	\begin{proof}
		Solving $Af = 0$, that is
		\begin{align*}
			f_s - \frac{d(s)}{2}f = 0,
		\end{align*}
		we obtain $f = C\sqrt{\rho}$. Hence,
		\begin{align*}
			\ker A = \text{span}~\{\sqrt{\rho}\}.
		\end{align*}
		Solving $A^*f = 0$, that is
		\begin{align*}
			-f_s - \frac{d(s)}{2}f = 0,
		\end{align*}
		one gets $f = C\rho^{-1/2}$. However, $\rho^{-1/2} \notin L^2(\R)$ since the integral
		\begin{align*}
			\int_{\R}\rho^{-1}ds = \int_0^\infty \phi(r)^{-2}r^{1-N}dr
		\end{align*}
		is divergent. So
		\begin{align*}
			\ker A^* = \{0\}.
		\end{align*}
		Then by \eqref{formdomainofH_0} and \eqref{formdomainoftildeH}, we have
		\begin{align*}
			\ker H_0 = \ker A, \quad \ker \widetilde{H} = \ker A^*,
		\end{align*}
		completing the proof.
	\end{proof}

	Using polar decomposition, one can check that the operator $H_0|_{(\ker H_0)^{\bot}}$ is unitarily equivalent to the operator $\widetilde{H}$, see, for example, \cite[Theorem 8.6]{Tes}. Hence
	\begin{align*}
		\sigma\left(H_0|_{(\ker H_0)^{\bot}}\right) = \sigma(\widetilde{H}).
	\end{align*}
	This enables us to obtain the following result.

	\begin{proposition} \label{propinfH}
		$\la_{\text{P}}(\rho) = \inf \sigma(\widetilde{H})$.
	\end{proposition}

	\begin{proof}
		For $g \in H^1(\rho ds)$ and $g \neq 0$, we set $f = \sqrt{\rho}g$. Clearly, $f \in L^2(\R)$ and
		\begin{align*}
			\int_{\R}g^2 \rho ds = \int_{\R}f^2 ds.
		\end{align*}
		Moreover,
		\begin{align*}
			\int_{\R}g \rho ds = \int_{\R}f\sqrt{\rho} ds,
		\end{align*}
		thus the condition $\int_{\R}g \rho ds = 0$ is equivalent to $f \in (\ker H_0)^{\bot}$. Noticing
		\begin{align*}
			\int_{\R}g_s^2\rho ds = \|Af\|_{L^2(\R)}^2 = \mathfrak{h}_0(f),
		\end{align*}
		we get
		\begin{align*}
			\la_{\text{P}}(\rho) = \inf_{0 \neq f \bot \sqrt{\rho}} \frac{\mathfrak{h}_0(f)}{\|f\|_{L^2(\R)}^2} = \inf \sigma\left(H_0|_{(\ker H_0)^{\bot}}\right).
		\end{align*}
		Recalling
		\begin{align*}
		\sigma\left(H_0|_{(\ker H_0)^{\bot}}\right) = \sigma(\widetilde{H}),
	\end{align*}
	we complete the proof.
	\end{proof}

\vskip0.3in

	Choose a test function such that $g_s = 1/k > 0$ and let
	\begin{align} \label{testfunction}
		\Phi := \sqrt{\rho}g_s = r^{(N-2)/2}v > 0.
	\end{align}
	Set $\tau = (N-2)/2$. Note that $-d_s(s)/2 = y_s + z_s$ and
	\begin{align*}
		\frac{\Phi_{ss}}{\Phi} = \frac{r(r\Phi_r)_r}{\Phi} = \tau^2 - (2\tau +1)y + \frac{r^2v_{rr}}{v} = (\tau - y)^2 - y_s.
	\end{align*}
	We obtain
	\begin{align} \label{HPhi}
		\frac{\widetilde{H}\Phi}{\Phi} = -\frac{\Phi_{ss}}{\Phi} + \frac{d(s)^2}{4} - \frac{d_s(s)}{2} = 2y_s + z_s - (\tau - y)^2 + \left(\frac{N}{2} - y - z\right)^2 = N-1 + r^2P - 2zT.
	\end{align}
	Now we set
	\begin{align} \label{def:G}
		\G := r^2P - 2zT = -rh\frac{d}{dr}\left(\frac{T}{h}\right).
	\end{align}
	Then,
	\begin{align*}
		\frac{\widetilde{H}\Phi}{\Phi} = N-1 + \G.
	\end{align*}

	\begin{proposition} \label{propG>01}
		Suppose $N \ge 2$ and $1 < q < p < 2^* -1$. When $2 \le N \le 5$, we further assume $p \le p_{\text{alg}}(N)$. Then $\G > 0$ in $r > 0$.
	\end{proposition}

	\begin{proof}
		By \eqref{eqextending} we have
		\begin{align} \label{expandingofz}
			z = -\frac{rh_r}{h} = \frac{P(0) - 2h_0}{N+2} r^2 + o(r^2).
		\end{align}
		Noticing
		\begin{align*}
			T(0) = N - y(0) - z(0) = N - 0 - \frac{h_0 - k_r(0)}{h_0} = N,
		\end{align*}
		we get
		\begin{align*}
			T = N + o(1).
		\end{align*}
		Moreover, $\xi(0) = -\ln v(0) = -\ln 1 = 0$ and
		\begin{align*}
			P(\xi(r)) = P(0) + o(1).
		\end{align*}
		Thus
		\begin{align*}
			\G = r^2P - 2zT = \frac{4Nh_0 - (N-2)P(0)}{N+2}r^2 + o(r^2).
		\end{align*}
		By Lemma \ref{lemestimate}, $k_r > J -h$. Sending $r$ to $0$,
		\begin{align*}
			k_r = rh_r + h \to h_0,
		\end{align*}
		and we have $h_0 \ge J(0) - h_0$, namely $h_0 \ge J(0)/2$. Together with
		\begin{align*}
			J(0) = (1-\theta)\frac{p-1}{p+1} + \theta \frac{q-1}{q+1} \ge \frac{1}{p+1} \left((1-\theta)(p-1) + \theta (q-1)\right) = \frac{P(0)}{p+1},
		\end{align*}
		we deduce that
		\begin{align*}
			h_0 \ge \frac{P(0)}{2(p+1)}.
		\end{align*}
		Since $P(0) > 0$ and $p < 2^*-1$, we obtain
		\begin{align*}
			\frac{4Nh_0 - (N-2)P(0)}{N+2} \ge \frac{1}{N+2}\left(\frac{2N}{p+1} - (N-2)\right)P(0) > 0,
		\end{align*}
		implying that $\G(r) > 0$ when $r > 0$ is small enough.

		Now we suppose by contradiction that $\G$ has a zero point in $r > 0$. Then there exists $R > 0$ such that
		\begin{align*}
			\G(r) > 0 ~~ (0 < r < R), \quad \G(R) = 0.
		\end{align*}
		We also have
		\begin{align*}
			\G_r(R) \le 0.
		\end{align*}
		At $r = R$, it is clear that
		\begin{align*}
			R^2P(\xi(R)) = 2z(R)T(R), \quad T(R) > 0.
		\end{align*}
		Using \eqref{ys} and \eqref{zs}, we get
		\begin{align*}
			r\G_r = \G_s & = r^2P_\xi\xi_r r_s + 2r^2P - 2z_sT - 2zT_s \\
			& = r^2P(2 - my) - 2T(r^2P - zT - y(2-z)) + 2z(r^2P - zT)
		\end{align*}
		From $R^2P(\xi(R)) = 2z(R)T(R)$ and $T = N - y -z$ one derives that
		\begin{align*}
			R\G_r(R) = 2T(R)\left(y(R)\left(2-m(\xi(R))z(R)\right) + z(R)(2z(R) - N +2)\right).
		\end{align*}
		Let
		\begin{align*}
			\mathcal{D} := y(2-mz) + z(2z - N +2).
		\end{align*}
		We aim to prove $\mathcal{D}(R) > 0$ to find a contradiction.
		
\vskip0.2in

		By Lemma \ref{lemsestimate} and using
		\begin{align*}
			k_r = rh_r + h = h(1-z)
		\end{align*}
		we obtain
		\begin{align*}
			h(1-z) > J - \beta h
		\end{align*}
		where $\beta$ is given by \eqref{def:beta}, namely
		\begin{align*}
			z < 1 + \beta - \frac{J}{h} < 1 + \beta.
		\end{align*}
		By $J \ge P/(p+1)$ and $r^2h = y$, we have
		\begin{align*}
			\frac{J}{h} \ge \frac{P}{(p+1)h} = \frac{r^2P}{(p+1)y}.
		\end{align*}
		At $r = R$, we get
		\begin{align*}
			\frac{J(\xi(R))}{h(R)} \ge \frac{2z(R)T(R)}{(p+1)y(R)},
		\end{align*}
		so
		\begin{align*}
			z(R) < 1 + \beta - \frac{2z(R)T(R)}{(p+1)y(R)}.
		\end{align*}
		Combining $T = N -y - z$ we deduce that
		\begin{align} \label{eqy(R)}
			y(R) > \frac{2z(R)(N-z(R))}{D_*}
		\end{align}
		where
		\begin{align*}
			D_* := (p+1)(1+\beta) - (p-1)z(R) > 2(1+\beta) > 0.
		\end{align*}

\noindent{\bf Case 1:} $(p-1)z(R) \le 2$.

		By $m \le p-1$, one gets
		\begin{align*}
			2-m(\xi(R))z(R) \ge 2 - (p-1)z(R) \ge 0.
		\end{align*}
		If $2z(R) - N + 2 \ge 0$, then $\mathcal{D}(R) \ge 0$. Moreover, if $\mathcal{D}(R) = 0$, we have
		\begin{align*}
			m(\xi(R))z(R) = 2, \quad 2z(R) = N -2.
		\end{align*}
		By Lemma \ref{lemestimate}, $z > 0$ when $r > 0$. Thus $N > 2$ in this case and
		\begin{align*}
			m(\xi(R)) = \frac{2}{z(R)} = \frac{4}{N-2},
		\end{align*}
		contradicting
		\begin{align*}
			m \le p-1 < 2^*-2 = \frac{4}{N-2}.
		\end{align*}
		This yields $\mathcal{D}(R) > 0$, a contradiction!

		If $2z(R) - N + 2 < 0$, using \eqref{eqy(R)} and
		\begin{align*}
			\mathcal{D}(R) & = y(R)(2-m(\xi(R))z(R)) + z(R)(2z(R) - N +2) \\
			& \ge y(R)(2-(p-1)z(R)) + z(R)(2z(R) - N +2),
		\end{align*}
		we get
		\begin{align*}
			\mathcal{D}(R) \ge \frac{z(R)}{D_*}E, \quad E = 2(N-z(R))(2-(p-1)z(R)) - D_*(N-2-2z(R)).
		\end{align*}
		Computing directly with \eqref{def:beta} yields
		\begin{align*}
			(p+3)E = E_0 + Kz(R)
		\end{align*}
		where
		\begin{align*}
			E_0 = 8(N + p + 1) - (p-1)(p+1)N(N-2), \quad K = (p-1)((p-1)(N-2) - 4) < 0.
		\end{align*}
		By $z(R) < 1 + \beta$ we have
		\begin{align*}
			E_0 + Kz(R) > E_0 + K(1+\beta).
		\end{align*}
		Computing directly again, one gets
		\begin{align*}
			(p+3)(E_0 + K(1+\beta)) = (4- (N-2)(p-1))(6N(p-1) + 8(N+2)) > 0,
		\end{align*}
		yielding $E > 0$. Thus $\mathcal{D}(R) > 0$, a contradiction!

\vskip0.22in
\noindent{\bf Case 2:} $(p-1)z(R) > 2$.

		At $r = R$,
		\begin{align*}
			T(R) = N - y(R) - z(R) > 0
		\end{align*}
		implying that $y(R) < N -z(R)$. Then
		\begin{align*}
			\mathcal{D}(R) & = y(R)(2-m(\xi(R))z(R)) + z(R)(2z(R) - N +2) \\
			& \ge y(R)(2-(p-1)z(R)) + z(R)(2z(R) - N +2) \\
			& > (N -z(R))(2-(p-1)z(R)) + z(R)(2z(R) - N +2)\\
			& =: R_p(z(R)).
		\end{align*}
		Note that
		\begin{align*}
			R_p(z(R)) = (p+1)z(R)^2 - Npz(R) + 2N
		\end{align*}
		is a quadratic polynomial with respect to $z(R)$. Let
		\begin{align*}
			\Delta_p = N^2p^2 - 8Np - 8N.
		\end{align*}
		The equation $\Delta_p = 0$ has the positive root
		\begin{align*}
			p_{\text{alg}}(N) = \frac{4 + 2\sqrt{2N + 4}}{N}.
		\end{align*}
		When $2 \le N \le 5$, the additional condition $p \le p_{\text{alg}}(N)$ enables us to obtain $\Delta_p \le 0$ and so $R_p(z(R)) \ge 0$. Thus we find the contradiction that $\mathcal{D}(R) > 0$.

		Finally, we notice that when $N \ge 6$, $p < (N+2)/(N-2) \le 2$ and
		\begin{align*}
			(p-1)z < (p-1)(1+\beta) = (p-1)\frac{Np - (N-4)}{p+3} < \frac{Np - (N-4)}{p+3} < 2.
		\end{align*}
		This means that Case 2 does not occur when $N \ge 6$ and we complete the proof.
	\end{proof}

\vskip0.3in

	\begin{proposition} \label{propG>02}
		Suppose $2 \le N \le 5$ and $1 < q < p < 2^* -1$. We further assume $p - q \le 1$. Then $\G > 0$ in $r > 0$.
	\end{proposition}

	\begin{proof}
		We argue by contradiction and assume that $R > 0$ is a zero point of $G$ such that
		\begin{align*}
			\G(r) > 0 ~~ (0 < r < R), \quad \G(R) = 0; \quad \quad \G_r(R) \le 0.
		\end{align*}
		At $\G = 0$, we have
		\begin{align*}
			P = \frac{2zT}{r^2} = 2z\frac{S}{r^2h} = 2z\frac{S}{y}.
		\end{align*}
		Thus
		\begin{align*}
			y = 2z\frac{S}{P} = 2z\frac{F - \eta}{P} = 2z(1-\ep)\frac{F}{P} = 2z(1-\ep)\frac{p+q-2-m}{(p-1)(q-1)}, \quad \text{ at } r = R,
		\end{align*}
		where
		\begin{align*}
			\ep := \frac{\eta}{F}; \quad \ep \in (0,1) ~~ \text{ at } r = R.
		\end{align*}
		Recall that
		\begin{align*}
			\mathcal{D} = y(2-mz) + z(2z - N +2).
		\end{align*}
		Similar to the proof of Proposition \ref{propG>01}, we aim to prove $\mathcal{D}(R) > 0$ to find a contradiction. Direct computations yield
		\begin{align} \label{eqofDatR}
			\frac{(p-1)(q-1)}{z} \mathcal{D} = (q-1)\left(4 - (N-2)(p-1)\right) & + 2(p-1-m)\left(2 - z(m+1-q)\right) \\
			& + 2\ep(p+q-2-m)(mz - 2)\quad \text{ at } r = R.
		\end{align}
		
		By Case 1 in the proof of Proposition \ref{propG>01}, we get $\mathcal{D}(R) > 0$ if $(p-1)z(R) \le 2$. Thus we assume that $(p-1)z(R) > 2$. When $N \ge 3$, we have
		\begin{align*}
			2z(R) - N + 2 > \frac{4}{p-1} - (N-2) > 0.
		\end{align*}
		When $N =2$, it is clear that
		\begin{align*}
			2z(R) - N + 2 = 2z(R) > 0.
		\end{align*}
		If $m(\xi(R))z(R) \le 2$, we derive $\mathcal{D}(R) > 0$. If $m(\xi(R))z(R) > 2$, using \eqref{eqofDatR}, $1 <q < p < 2^* -1$, $q-1 < m < p-1$, $z < 2$, and
		\begin{align*}
			m+1 - q < p - q \le 1,
		\end{align*}
		we conclude $\mathcal{D}(R) > 0$. The proof is complete.
	\end{proof}

	\section{Positivity of $\G$ when $N = 5$} \label{G>0whenN=5}

	\begin{proposition} \label{propG>03}
		Suppose $N = 5$ and $1 < q < p < 7/3$. Then $\G > 0$ in $r > 0$.
	\end{proposition}

	\begin{proof} We divide the proof into 6 steps.

\vskip0.12in

\noindent{\bf Step 1:} Let $\cR = 2y/3 - z$, then
		\begin{align} \label{eqcR}
			\cR(r) > 0 \quad \text{ in } r > 0.
		\end{align}

		Let
		\begin{align}
			I_2 = \int_{\R^5}v^2dx; \quad M_{\sigma} = \frac{\int_{\R^5}v^{\sigma+1}dx}{I_2}, ~ \sigma \in \{p,q\}; \quad \omega_{p} = 1-\theta, ~ \omega_{q} = \theta.
		\end{align}
		Using the Nehari identity
		\begin{align*}
			\int_{\R^5}|\nabla v|^2dx + \eta I_2 = (1-\theta)\int_{\R^5}v^{p+1}dx + \theta \int_{\R^5}v^{q+1}dx
		\end{align*}
		and the Pohozaev identity
		\begin{align*}
			\frac{3}{2}\int_{\R^5}|\nabla v|^2dx + \frac{5}{2}\eta I_2 = \frac{5(1-\theta)}{p+1}\int_{\R^5}v^{p+1}dx + \frac{5\theta}{q+1}\int_{\R^5}v^{q+1}dx,
		\end{align*}
		we get
		\begin{align} \label{eqofeta}
			\eta = \sum_{\sigma \in \{p,q\}} \omega_\sigma \frac{7 - 3\sigma}{2(\sigma+1)}M_\sigma.
		\end{align}
		Since $v(0) = 1$ and $0 < v(r) < 1$ in $r > 0$, we have $0 < M_\sigma < 1$. By \eqref{eqextending}, one obtains
		\begin{align*}
			y = r^2h = h_0r^2 + o(r^2).
		\end{align*}
		Together with \eqref{expandingofz} we deduce that
		\begin{align*}
			\cR = \frac23 h_0r^2 - \frac{P(0) - 2h_0}{7} r^2 + o(r^2) = \frac{20h_0 - 3P(0)}{21}r^2 + o(r^2).
		\end{align*}
		Using $h_0 = -v_{rr}(0) = (1-\eta)/5$ and \eqref{eqofeta}, one gets
		\begin{align*}
			20h_0 - 3P(0) & = 4 - 4\eta - 3\sum_{\sigma}(\sigma - 1)\omega_\sigma \\
			& = \sum_\sigma (7 - 3\sigma)\omega_\sigma - 4\eta \\
			& = \sum_\sigma (7-3\sigma)\omega_\sigma \left(1 - \frac{2M_\sigma}{\sigma + 1}\right) \\
			& > \sum_\sigma (7-3\sigma)\omega_\sigma \left(1 - \frac{2}{\sigma + 1}\right) \\
			& > 0.
		\end{align*}
		Thus $\cR(r) > 0$ when $r > 0$ is small enough.

		Now we suppose by contradiction that $\cR$ has a zero point in $r > 0$. There exists $r_1 > 0$ such that
		\begin{align*}
			\cR(r) > 0 ~~ (0 < r < r_1), \quad \cR(r_1) = 0.
		\end{align*}
		We also have
		\begin{align*}
			\cR_s(r_1) = r_1\cR_r(r_1) \le 0.
		\end{align*}
		At $r = r_1$,
		\begin{align*}
			z(r_1) = \frac{2}{3}y(r_1).
		\end{align*}
		From $z < 2$ it follows $y(r_1) < 3$. Moreover,
		\begin{align*}
			T(r_1) = 5 - y(r_1)- z(r_1) > 0.
		\end{align*}
		The equation \eqref{ys} yields
		\begin{align*}
			y_s(r_1) = y(r_1)(2-z(r_1)) = \frac{2}{3}y(r_1)(3-y(r_1)).
		\end{align*}
		One can see that
		\begin{align*}
			\frac{J}{P} = \frac{\sum_\sigma \frac{\sigma-1}{\sigma +1}\omega_\sigma e^{-(\sigma - 1)\xi}}{\sum_\sigma (\sigma - 1)\omega_\sigma e^{-(\sigma - 1)\xi}} > \frac{3}{10}.
		\end{align*}
		By Lemma \ref{lemestimate},
		\begin{align*}
			y_s = ry_r = r^2k_r + r^2h > r^2J = r^2P \frac{J}{P} > \frac{3}{10}r^2P.
		\end{align*}
		At $r = r_1$ we obtain
		\begin{align*}
			r_1^2P(\xi(r_1)) < \frac{10}{3}y_s(r_1) = \frac{20}{9}y(r_1)(3-y(r_1)),
		\end{align*}
		and
		\begin{align*}
			\cR_s(r_1) & = \frac23 y_s(r_1) - z_s(r_1) \\
			& = \frac{4}{9}y(r_1)(3-y(r_1)) - r_1^2P(\xi(r_1)) + z(r_1)T(r_1) + y(r_1)(2-z(r_1)) \\
			& = \frac{20}{9}y(r_1)(3-y(r_1)) - r_1^2P(\xi(r_1)) \\
			& > 0,
		\end{align*}
		a contradiction. This proves \eqref{eqcR}.

\vskip0.12in
\noindent{\bf Step 2:} Take
		\begin{align} \label{def:W_0(r)}
			W_0(r) = \left(1 + \frac{h_0r^2}{3}\right)^{-3/2} = 1 - \frac{h_0}{2}r^2 + O(r^4) \in L^2(\R^5)
		\end{align}
		and define
		\begin{align*}
			M_\sigma^{(0)} := \frac{\int_{\R^5}W_0(|x|)^{\sigma +1}dx}{\int_{\R^5}W_0(|x|)^2dx},
		\end{align*}
		then
		\begin{align} \label{M>M0}
			M_\sigma = \frac{\int_{\R^5}v^{\sigma +1}dx}{\int_{\R^5}v^2dx} > M_\sigma^{(0)}, \quad \sigma \in \{p,q\}.
		\end{align}
Recall that $\xi = -\ln v(r)$ is strictly increasing in $r > 0$. Thus we can view $y$ as a function of $\xi$ and
		\begin{align}
			y_\xi = \frac{y_s}{\xi_s} = \frac{y_s}{r\xi_r} = \frac{y_s}{y} = 2 - z > 2 - \frac23 y.
		\end{align}
		Let
		\begin{align*}
			y_0(\xi) = 3\left(1 - e^{-2\xi/3}\right),
		\end{align*}
		which satisfies
		\begin{align*}
			\frac{dy_0}{d\xi} = 2 - \frac{2}{3}y_0, \quad y_0(0) = 0.
		\end{align*}
		It follows that
		\begin{align*}
			\frac{d}{d\xi}\left(e^{2\xi/3}(y - y_0)\right) > 0.
		\end{align*}
		When $\xi = 0$, $v(r) = 1$ and so $r = 0$. At this point, the value of $y$ is $0$. Therefore, we obtain
		\begin{align*}
			y > y_0 \quad \text{ for all } \xi > 0.
		\end{align*}
Let $r(\xi)$ and $r_0(\xi)$ be determined by
		\begin{align*}
			v(r(\xi)) = e^{-\xi}, \quad W_0(r_0(\xi)) = e^{-\xi}
		\end{align*}
		respectively. Then
		\begin{align*}
			\frac{d}{d\xi} \ln r(\xi) = \frac{r_\xi(\xi)}{r(\xi)} = \frac{1}{r(\xi)\xi_r(r(\xi))} = \frac{1}{y(\xi)}.
		\end{align*}
		Moreover,
		\begin{align*}
			2\ln r_0(\xi) = \ln 3(e^{2\xi/3} - 1) - \ln h_0,
		\end{align*}
		and we have
		\begin{align*}
			\frac{d}{d\xi} \ln r_0(\xi) = \frac{e^{2\xi/3}}{3(e^{2\xi/3} - 1)} = \frac{1}{y_0(\xi)}.
		\end{align*}
		Then, one gets
		\begin{align*}
			\frac{d}{d\xi} \ln \frac{r(\xi)}{r_0(\xi)} = \frac{1}{y(\xi)} - \frac{1}{y_0(\xi)} < 0 \quad \text{ in } \xi > 0.
		\end{align*}
		By
		\begin{align*}
			1 - \xi + O(\xi^2) = e^{-\xi} = v(r(\xi)) = 1 - \frac{h_0}{2}r(\xi)^2 + O(r(\xi)^4),
		\end{align*}
		we derive
		\begin{align*}
			r(\xi) \sim \sqrt{\frac{2\xi}{h_0}} \quad \text{ as } \xi \to 0.
		\end{align*}
		Similarly, we have
		\begin{align*}
			r_0(\xi) \sim \sqrt{\frac{2\xi}{h_0}} \quad \text{ as } \xi \to 0.
		\end{align*}
		Thus
		\begin{align*}
			\frac{r(\xi)}{r_0(\xi)} \to 1 \quad \text{ as } \xi \to 0.
		\end{align*}
Define
		\begin{align*}
			\mathcal{M}_{v}(\kappa) := \int_{\R^5}v^\kappa dx, ~~ \mathcal{M}_{W_0}(\kappa) := \int_{\R^5}W_0^\kappa dx, \quad \kappa \ge 2.
		\end{align*}
		By the layer cake representation, see for example \cite{LL}, we deduce that
		\begin{align*}
			& \mathcal{M}_{v}(\kappa) = \kappa\int_0^\infty e^{-\kappa\xi}\big|\big\{x \in \R^5: v(x) > e^{-\xi} \big\}\big|d\xi, \\
			& \mathcal{M}_{W_0}(\kappa) = \kappa\int_0^\infty e^{-\kappa\xi}\big|\big\{x \in \R^5: W_0(x) > e^{-\xi} \big\}\big|d\xi.
		\end{align*}
		We introduce the probability measure
		\begin{align*}
			d\mu_\kappa(\xi) = \frac{e^{-\kappa\xi}\big|\big\{x \in \R^5: W_0(x) > e^{-\xi} \big\}\big|d\xi}{\int_0^\infty e^{-\kappa t}\big|\big\{x \in \R^5: W_0(x) > e^{-t} \big\}\big|dt}.
		\end{align*}
		Let
		\begin{align*}
			a(\xi) := \frac{\big|\big\{x \in \R^5: v(x) > e^{-\xi} \big\}\big|}{\big|\big\{x \in \R^5: W_0(|x|) > e^{-\xi} \big\}\big|} = \left(\frac{r(\xi)}{r_0(\xi)}\right)^5.
		\end{align*}
		The function $a(\xi)$ is strictly decreasing in $\xi > 0$ and $a$ extends continuously to $\xi = 0$ with $a(0) = 1$. By the strict Chebyshev integral inequality, see \cite{Jak}, we get
		\begin{align*}
			\text{Cov}_{\mu_\kappa}(\xi,a(\xi)) = \int \xi a(\xi)d\mu_\kappa - \left(\int \xi d\mu_\kappa\right)\left(\int a(\xi) d\mu_\kappa\right) < 0.
		\end{align*}
		Direct computations yield
		\begin{align*}
			\frac{d}{d\kappa}\ln \mathcal{M}_{W_0}(\kappa) = \frac{1}{\kappa} - \int \xi d\mu_\kappa, \quad \frac{d}{d\kappa}\ln \mathcal{M}_{v}(\kappa) = \frac{1}{\kappa} - \frac{\int \xi a(\xi) d\mu_\kappa}{\int a(\xi) d\mu_\kappa}.
		\end{align*}
		Thus
		\begin{align*}
			\frac{d}{d\kappa}\ln \mathcal{M}_{v}(\kappa) > \frac{d}{d\kappa}\ln \mathcal{M}_{W_0}(\kappa).
		\end{align*}
		Integrating for $\kappa \in [2,1+\sigma]$, we get \eqref{M>M0}.  Define
		\begin{align*}
			\bar\xi_0 := \frac{\int_{\R^5}W_0^2(-\ln W_0)dx}{\int_{\R^5}W_0^2dx}.
		\end{align*}
		Let $t = h_0 r^2/3$, then
		\begin{align*}
			W_0^2 = (1+t)^{-3}, \quad -\ln W_0 = \frac{3}{2}\ln(1+t),
		\end{align*}
		and
		\begin{align*}
			\bar\xi_0 = \frac{3}{2}\frac{\int_0^\infty t^{3/2}(1+t)^{-3}\ln(1+t)dt}{\int_0^\infty t^{3/2}(1+t)^{-3}dt} = \frac{9}{4} + 3\ln 2.
		\end{align*}

\vskip0.12in
\noindent{\bf Step 3:} Let
		\begin{align*}
			\Gamma = \frac{2S}{P}, \quad d_0 = \frac{4}{3} - m > 0, \quad \mathscr{C} = \frac{2\Gamma/3 - 1}{d_0},
		\end{align*}
		then
		\begin{align} \label{C<}
			\mathscr{C} < \bar\xi_0 + \frac{1}{2} - \xi.
		\end{align}
We write
		\begin{align*}
			W_\sigma = \omega_\sigma e^{-(\sigma-1)\xi}, \quad \sigma \in \{p,q\},
		\end{align*}
		and define
		\begin{align*}
			a_\sigma := (\sigma - 1)\left(\frac{7}{3} - \sigma\right)W_\sigma, \quad \mathscr{C}_\sigma := \frac{1}{\sigma-1}\left(1 - \frac{2M_\sigma}{\sigma + 1}e^{(\sigma - 1)\xi}\right).
		\end{align*}
		Note that
		\begin{align*}
			d_0 = \frac{4}{3} - m = \frac{4}{3} - \frac{\sum_\sigma(\sigma - 1)^2W_\sigma}{P} = \frac{\sum_\sigma a_\sigma}{P}.
		\end{align*}
		Moreover, using \eqref{eqofeta} we get
		\begin{align*}
			\frac{2}{3}\Gamma -1 = \frac{\sum_\sigma a_\sigma \mathscr{C}_\sigma}{P}.
		\end{align*}
		Thus
		\begin{align} \label{eqofscrC}
			\mathscr{C} = \frac{2\Gamma/3 - 1}{d_0} = \frac{\sum_\sigma a_\sigma \mathscr{C}_\sigma}{d_0P} = \frac{\sum_\sigma a_\sigma \mathscr{C}_\sigma}{\sum_\sigma a_\sigma}.
		\end{align}
Let $\xi_0 = - \ln W_0$ and
		\begin{align*}
			d\mu_0 := \frac{W_0^2dx}{\int_{\R^5}W_0^2dx},
		\end{align*}
		then
		\begin{align*}
			\bar\xi_0 = \int \xi_0 d\mu_0, \quad M_\sigma^{(0)} = \int e^{-(\sigma - 1)\xi_0}d\mu_0.
		\end{align*}
		Using Jensen's inequality we obtain
		\begin{align*}
			\ln\left(\int e^{-(\sigma - 1)\xi_0}d\mu_0\right) \ge -(\sigma - 1)\int \xi_0d\mu_0 = -(\sigma - 1)\bar \xi_0,
		\end{align*}
		namely
		\begin{align*}
			M_\sigma^{(0)} \ge e^{-(\sigma - 1)\bar \xi_0}.
		\end{align*}
		By \eqref{M>M0} and
		\begin{align*}
			\frac{2}{\sigma + 1} = \frac{1}{1 + (\sigma-1)/2} \ge e^{-(\sigma - 1)/2},
		\end{align*}
		we derive that
		\begin{align*}
			\frac{2M_\sigma}{\sigma + 1}e^{(\sigma - 1)\xi} > e^{-(\sigma - 1)(\bar\xi_0 + 1/2 - \xi)}.
		\end{align*}
		Setting $\tau = (\sigma - 1)(\bar\xi_0 + 1/2 - \xi)$ and using
		\begin{align*}
			1 - e^{-\tau} \le \tau,
		\end{align*}
		we conclude that
		\begin{align*}
			\mathscr{C}_\sigma = \frac{1}{\sigma-1}\left(1 - \frac{2M_\sigma}{\sigma + 1}e^{(\sigma - 1)\xi}\right) < \bar\xi_0 + \frac12 - \xi.
		\end{align*}
		Together with \eqref{eqofscrC} we obtain \eqref{C<}.

		\vskip0.12in
\noindent{\bf Step 4:} Recall that $\G(r) > 0$ when $r > 0$ is small enough. From this step we suppose by contradiction that $\G$ has a zero point in $r > 0$. Thus there exists $r_* > 0$ such that
		\begin{align*}
			\G(r) > 0 ~~ (0 < r < r_*), \quad \G(r_*) = 0.
		\end{align*}
		We also have
		\begin{align*}
			\G_s(r_*) = r_*\G_r(r_*) \le 0.
		\end{align*}
At $r = r_*$, we have that
		\begin{align}
			& r_*^2P(\xi(r_*)) = 2z(r_*)T(r_*) = 2z(r_*)\frac{S(\xi(r_*))}{h(r_*)}, \quad T(r_*) > 0, \\
			& y(r_*) = r_*^2h(r_*) = \frac{2z(r_*)S(\xi(r_*))}{P(\xi(r_*))} = z(r_*)\Gamma(\xi(r_*)) \label{eqy=zgamma}.
		\end{align}
		Recall that
		\begin{align*}
			\mathcal{D} = y(2 - mz) + z(2z - 3).
		\end{align*}
		We have proved
		\begin{align*}
			\G_s(r_*) = 2T(r_*)\mathcal{D}(r_*).
		\end{align*}
		By
		\begin{align*}
			J - h < k_r = rh_r + h = h(1 - z),
		\end{align*}
		we get
		\begin{align*}
			J < h(2-z).
		\end{align*}
		At $r = r_*$,
		\begin{align*}
			\frac{J(\xi(r_*))}{h(r_*)} = \frac{J(\xi(r_*))}{P(\xi(r_*))}\frac{P(\xi(r_*))}{h(r_*)} = \frac{J(\xi(r_*))}{P(\xi(r_*))}\frac{r_*^2P(\xi(r_*))}{y(r_*)} = \frac{J(\xi(r_*))}{P(\xi(r_*))}\frac{2z(r_*)T(r_*)}{y(r_*)}.
		\end{align*}
		Using $J/P > 3/10$, one gets
		\begin{align*}
			\frac{3z(r_*)T(r_*)}{5y(r_*)} < 2 - z(r_*),
		\end{align*}
		namely
		\begin{align*}
			3z(r_*)(5 - z(r_*)) < 2y(r_*)(5 - z(r_*))
		\end{align*}
		where we have used $T = 5 - y -z$. Since $z < 2$, we derive that
		\begin{align*}
			y(r_*) > \frac{3}{2}z(r_*).
		\end{align*}
		If $m(\xi(r_*))z(r_*) \le 2$, then
		\begin{align*}
			\mathcal{D}(r_*) & = y(r_*)(2 - m(\xi(r_*))z(r_*)) + z(r_*)(2z(r_*) - 3) \\
			& \ge \frac{3}{2}z(r_*)(2 - m(\xi(r_*))z(r_*)) + z(r_*)(2z(r_*) - 3) \\
			& = \frac{z(r_*)^2}{2}(4 - 3m(\xi(r_*))) \\
			& > 0,
		\end{align*}
		a contradiction! Thus
		\begin{align} \label{mz>2}
			m(\xi(r_*))z(r_*) > 2.
		\end{align}
		Repeating the argument in Case 2 of the proof of Proposition \ref{propG>01}, we obtain
		\begin{align*}
			m(\xi(r_*)) > p_{\text{alg}}(5) - 1 = \frac{2\sqrt{14} - 1}{5}.
		\end{align*}
By \eqref{eqy=zgamma} and
		\begin{align*}
			\Gamma = \frac{3}{2}(1 + d_0\mathscr{C}),
		\end{align*}
		we have
		\begin{align*}
			\mathcal{D}(r_*) & = z(r_*)\Gamma(\xi(r_*))(2 - m(\xi(r_*))z(r_*)) + z(r_*)(2z(r_*) - 3) \\
			& = z(r_*) \left(\frac{3}{2}(1 + d_0(\xi(r_*))\mathscr{C}(\xi(r_*)))(2 - m(\xi(r_*))z(r_*)) + 2z(r_*) - 3\right) \\
			& = \frac{3}{2}d_0(\xi(r_*))z(r_*)\left(\mathscr{C}(\xi(r_*))(2-m(\xi(r_*))z(r_*)) + z(r_*)\right).
		\end{align*}
		Since $\mathcal{D}(r_*) \le 0$, it holds that
		\begin{align} \label{z+Cle0}
			\mathscr{C}(\xi(r_*))(2-m(\xi(r_*))z(r_*)) + z(r_*) \le 0.
		\end{align}
		Together with \eqref{mz>2} we get
		\begin{align} \label{scrC>0atr*}
			\mathscr{C}(\xi(r_*)) \ge \frac{z(r_*)}{m(\xi(r_*))z(r_*) - 2} > 0.
		\end{align}
		Let
		\begin{align*}
			\Xi(r) := m\mathscr{C}.
		\end{align*}
		Then
		\begin{align*}
			\Xi(r_*) = m(\xi(r_*))\mathscr{C}(\xi(r_*)) \ge \frac{m(\xi(r_*))z(r_*)}{m(\xi(r_*))z(r_*) - 2}.
		\end{align*}
		Since
		\begin{align*}
			2 < m(\xi(r_*))z(r_*) < \frac{8}{3},
		\end{align*}
		one gets
		\begin{align*}
			\Xi(r_*) > 4.
		\end{align*}
		We further prove
		\begin{align*}
			\Xi(r) > 1 \quad \text{ for } 0 < r \le r_*.
		\end{align*}
		Suppose by contradiction that there exists $r^* \in (0,r_*)$ such that
		\begin{align*}
			\Xi(r) > 1 ~~ (r^* < r \le r_*); \quad \Xi(r^*) = 1.
		\end{align*}
		We also have
		\begin{align*}
			\Xi_s(r^*) = r^*\Xi_r(r^*) \ge 0.
		\end{align*}
		Direct computations yield
		\begin{align*}
			m_s = -y(p-1-m)(m-q+1) < 0, \quad \Gamma_s = y(m\Gamma -2).
		\end{align*}
		Together with
		\begin{align*}
			\Gamma = \frac{3}{2}(1 + d_0\mathscr{C}), \quad \Xi = m\mathscr{C}, \quad d_0 = \frac{4}{3} - m,
		\end{align*}
		we get
		\begin{align*}
			\Xi_s = y\left(m(\Xi - 1) - \frac{4/3(p-1-m)(m-q+1)}{md_0}\Xi\right).
		\end{align*}
		Particularly,
		\begin{align*}
			\Xi_s(r^*) < 0,
		\end{align*}
		a contradiction! Observe that
		\begin{align} \label{mg-2>0}
			m\Gamma - 2 = \frac{3}{2}d_0(\Xi-1) > 0 \quad \text{ for } 0 < r \le r_*.
		\end{align}
		Since $m_s = rm_r < 0$, we also have
		\begin{align} \label{m>23}
			m(\xi(r)) \ge m(\xi(r_*)) > p_{\text{alg}}(5) - 1 = \frac{2\sqrt{14} - 1}{5} > \frac23 \quad \text{ for } 0 < r \le r_*.
		\end{align}

	\vskip0.12in
\noindent{\bf Step 5:} In this step we aim to prove
		\begin{align} \label{talenti}
			\xi > -\frac{3}{2}\ln\left(1 - \frac{z}{2}\right) \quad \text{ at }r = r_*.
		\end{align}
We first prove
		\begin{align} \label{cRs>0}
			\cR_s(r) > 0 \quad \text{ for } 0 < r < r_*.
		\end{align}
		Noticing
		\begin{align*}
			\cR_s = r\cR_r = 2cr^2 + o(r^2), \quad c = \frac{20h_0 - 3P(0)}{21} > 0,
		\end{align*}
		we have $\cR_s > 0$ when $r > 0$ is sufficiently small. Arguing by contradiction, we assume that there exists $r_1 \in (0,r_*)$ such that
		\begin{align*}
			\cR_s(r) > 0 ~~ (0 < r < r_1), \quad \cR_s(r_1) = 0.
		\end{align*}
		We also have
		\begin{align*}
			\cR_{ss}(r_1) \le 0.
		\end{align*}
Let
		\begin{align*}
			\pi = \frac{y}{z}.
		\end{align*}
		At $r = r_1 < r_*$,
		\begin{align*}
			\G(r_1) = r_1^2P(\xi(r_1))\left(1 - \frac{z(r_1)\Gamma(\xi(r_1))}{y(r_1)}\right) > 0,
		\end{align*}
		namely
		\begin{align*}
			\pi(r_1) > \Gamma(\xi(r_1)).
		\end{align*}
		Together with \eqref{mg-2>0} we get
		\begin{align} \label{eqofpi>}
			\pi(r_1) > \frac{2}{m(\xi(r_1))} > \frac{3}{2}.
		\end{align}
		Set $X = r^2P$. At $\cR_s = 0$, combining
		\begin{align*}
			\cR_s = \frac{5}{3}y(2-z) - X + zT = 0
		\end{align*}
		and
		\begin{align*}
			T = \frac{X\Gamma}{2y},
		\end{align*}
		we solve that
		\begin{align}
			& X = \frac{10z\pi^2(2-z)}{3(2\pi - \Gamma)}, \\
			& T = \frac{5\pi \Gamma(2-z)}{3(2\pi - \Gamma)}, \label{solvingT}
		\end{align}
		which are well-defined at $r = r_1$ since
		\begin{align} \label{2pi-G>0}
			2\pi - \Gamma > \Gamma > 0 \quad \text{ at } r = r_1.
		\end{align}
		Then, direct computations yield
		\begin{align*}
			\cR_{ss}(r_1) = \frac{10z(r_1)\pi(r_1)(2-z(r_1))}{9(2\pi(r_1) - \Gamma(\xi(r_1)))}\mathcal{B}(r_1),
		\end{align*}
		where
		\begin{align*}
			\mathcal{B} = \Gamma(2\pi - 3) + z \left(3\Gamma + \pi\left((3m-2)\pi - 6\right)\right).
		\end{align*}
If
		\begin{align*}
			3\Gamma + \pi\left((3m-2)\pi - 6\right) \ge 0 \quad \text{ at } r = r_1,
		\end{align*}
		using \eqref{eqofpi>}, $\Gamma > 0$ and $z > 0$, we obtain $\mathcal{B}(r_1) > 0$, a contradiction! If
		\begin{align*}
			3\Gamma + \pi\left((3m-2)\pi - 6\right) < 0 \quad \text{ at } r = r_1,
		\end{align*}
		by \eqref{mg-2>0}, \eqref{eqofpi>} and $z < 2$, at $r = r_1$ one gets
		\begin{align} \label{B1}
			\mathcal{B} & > \Gamma(2\pi - 3) + 2\left(3\Gamma + \pi\left((3m-2)\pi - 6\right)\right) \\
			& = \Gamma(2\pi + 3) + 2(3m-2)\pi^2 - 12\pi \\
			& > \frac{2}{m}(2\pi + 3) + 2(3m-2)\pi^2 - 12\pi \\
			&= \frac{2}{m}(m\pi - 1)\left((3m-2)\pi - 3\right).
		\end{align}
		Combining \eqref{solvingT} and
		\begin{align*}
			T = 5 - y - z = 5 - z(\pi + 1),
		\end{align*}
		we solve that
		\begin{align*}
			(2-z)\left(6\pi^2 + 6\pi - \Gamma(8\pi + 3)\right) = 3(2\pi - \Gamma)(2\pi - 3).
		\end{align*}
		Since $z < 2$, together with \eqref{eqofpi>} and \eqref{2pi-G>0} we deduce that
		\begin{align*}
			6\pi^2 + 6\pi - \Gamma(8\pi + 3) > 0 \quad \text{ at } r = r_1.
		\end{align*}
		Then we get
		\begin{align*}
			z = \frac{5\left(6\pi - \Gamma(2\pi + 3)\right)}{6\pi^2 + 6\pi - \Gamma(8\pi + 3)} \quad \text{ at } r = r_1.
		\end{align*}
		Recalling $z > 0$, we have
		\begin{align*}
			\Gamma < \frac{6\pi}{2\pi + 3} \quad \text{ at } r = r_1.
		\end{align*}
		Together with \eqref{mg-2>0} one gets
		\begin{align*}
			\frac{2}{m} < \frac{6\pi}{2\pi + 3} \quad \text{ at } r = r_1,
		\end{align*}
		namely
		\begin{align} \label{B2}
			(3m - 2)\pi > 3 \quad \text{ at } r = r_1.
		\end{align}
		By \eqref{m>23} we also have
		\begin{align} \label{B3}
			m\pi - 1 > \frac{3m}{3m-2} - 1 > 0 \quad \text{ at } r = r_1.
		\end{align}
		Combining \eqref{B1}, \eqref{B2} and \eqref{B3}, we conclude that $\mathcal{B}(r_1) > 0$, a contradiction! Thus \eqref{cRs>0} holds true.

		Let
		\begin{align*}
			\mathcal{Q} := \xi + \frac{3}{2}\ln\left(1 - \frac{z}{2}\right).
		\end{align*}
		We have
		\begin{align*}
			\mathcal{Q}_s = y - \frac{3z_s}{2(2-z)}.
		\end{align*}
		Since
		\begin{align*}
			\cR_s = \frac{2}{3}y(2-z) - z_s,
		\end{align*}
		we get
		\begin{align*}
			\mathcal{Q}_s = \frac{3\cR_s}{2(2-z)} > 0 \quad \text{ in } r \in (0,r_*),
		\end{align*}
		yielding
		\begin{align*}
			\mathcal{Q}(r_*) > \mathcal{Q}(0) =0.
		\end{align*}
		This is \eqref{talenti}.

\vskip0.12in
\noindent{\bf Step 6:} We find a contradiction in this step to complete the proof.

		By \eqref{C<}, \eqref{scrC>0atr*} and \eqref{talenti}
		\begin{align*}
			0 < \mathscr{C} < \bar\xi_0 + \frac{1}{2} + \frac{3}{2}\ln\left(1 - \frac{z}{2}\right) \quad \text{ at }r = r_*.
		\end{align*}
		Let
		\begin{align*}
			\mathscr{C}_0 = \bar\xi_0 + \frac{1}{2} + \frac{3}{2}\ln\left(1 - \frac{z(r_*)}{2}\right) > 0.
		\end{align*}
		Then,
		\begin{align} \label{z+C>}
			z + \mathscr{C}(2-mz) > z + \mathscr{C}\left(2-\frac{4}{3}z\right) > z + \mathscr{C}_0\left(2-\frac{4}{3}z\right) \quad \text{ at }r = r_*,
		\end{align}
		where we have used $m < 4/3$ and
		\begin{align*}
			z(r_*) > \frac{2}{m(\xi(r_*))} > \frac{3}{2}.
		\end{align*}
		Let
		\begin{align*}
			\zeta = 4\left(1 - \frac{z(r_*)}{2}\right) \in (0,1).
		\end{align*}
		Recalling
		\begin{align*}
			\bar\xi_0 = \frac{9}{4} + 3\ln 2,
		\end{align*}
		we get
		\begin{align*}
			\mathscr{C}_0 = \frac{11}{4} + 3\ln 2 + \frac{3}{2}\ln \frac{\zeta}{4} = \frac{11}{4} + \frac{3}{2}\ln \zeta.
		\end{align*}
		Thus
		\begin{align*}
			z(r_*) + \mathscr{C}_0\left(2-\frac{4}{3}z(r_*)\right) = 2 - \frac{\zeta}{2} + \left(\frac{11}{4} + \frac{3}{2}\ln \zeta\right)\frac{2}{3}(\zeta - 1) = \frac{1}{6} + \frac{4}{3}\zeta + (\zeta - 1)\ln \zeta > 0.
		\end{align*}
		Together with \eqref{z+C>}, we obtain
		\begin{align*}
			z + \mathscr{C}(2-mz) > 0 \quad \text{ at } r = r_*,
		\end{align*}	
		contradicting \eqref{z+Cle0}. The proof is complete.
	\end{proof}

\vskip0.36in

	\section{The Morse index of any positive solution is $1$} \label{secMorse}

	In this section we complete the proof of Theorem \ref{thmMorse}, that is, for a solution $u$ of \eqref{doublepower}, we prove $m(u) = 1$. Equivalently, it suffices to prove $m(v) = 1$ where $v$ is the rescaled solution given by \eqref{scalingv}.

	We first prove $m_{rad}(v) \le 1$.
	
	When $N \ge 3$, by arguments in Section \ref{sectransformation}, showing $\la_{\text{P}}(\rho) \ge N-1$ will complete the proof of this step where $\la_{\text{P}}(\rho)$ is defined in \eqref{def:laPrho}. By Proposition \ref{propinfH}, this is to prove $\inf \sigma(\widetilde{H}) \ge N-1$. Let $\Phi$ be the test function given in \eqref{testfunction}. Using the identity \eqref{HPhi}, for any $f \in C_c^\infty(\R)$ we get
	\begin{align*}
		\big\langle (\widetilde{H}- (N-1))f,f\big\rangle_{L^2(\R)} = \int_{\R}\Phi^2\left|\left(\frac{f}{\Phi}\right)_s\right|^2 ds + \int_{\R}\G f^2 ds \ge \int_{\R}\G f^2 ds.
	\end{align*}
	Under conditions of Theorem \ref{thmMorse}, by Proposition \ref{propG>01}, Proposition \ref{propG>02} and Proposition \ref{propG>03}, we have $\G > 0$. Thus
	\begin{align*}
		\inf \sigma(\widetilde{H}) \ge N-1.
	\end{align*}
When $N = 2$, we still have $\G > 0$. By \eqref{ys}, \eqref{zs} and \eqref{def:G}, we get
	\begin{align*}
		r(2y+z)_r = (2y + z)_s = y(2-z) + r^2P - zT = y(2-z) + \frac{r^2P + \G}{2} > 0.
	\end{align*}
	As $r \to 0$, $2y + z \to 0$. As $r \to \infty$, by Lemma \ref{lemlimitofk} we have $y \sim \sqrt{\eta}r \to \infty$. Thus there exists a unique $r_0 > 0$ such that
	\begin{align*}
		2y(r_0) + z(r_0) = N-1.
	\end{align*}
	Take a function $g$ such that $g_s = 1/k > 0$ and $g(r_0) = 0$. It is clear that
	\begin{align*}
		g < 0 ~~ \text{ in } r < r_0, \quad\quad g > 0 ~~ \text{ in } r > r_0.
	\end{align*}
	Moreover, observing
	\begin{align*}
		\frac{d}{ds}\ln \left(\frac{\rho}{k}\right) = \frac{\rho_s}{\rho} - \frac{rk_r}{k} = N-1 -2y - z,
	\end{align*}
	we deduce that $(\rho g_s)_s(s_0) = 0$ with $s_0 = \ln r_0$, and so
	\begin{align*}
		\left(L_\rho - (N-1)\right)g(r_0) = 0,
	\end{align*}
	where $L_\rho$ is defined in \eqref{def:Lrho}. Note that
	\begin{align*}
		g_{ss} = \left(\frac{1}{k}\right)_s = -\frac{rk_r}{k^2}=-\frac{r^2h_r + k}{k^2} = (z-1)g_s.
	\end{align*}
	Then we have
	\begin{align*}
		\left(\left(L_\rho - (N-1)\right)g\right)_s = \left(-(N-1-2y-z)g_s - (N-1)g\right)_s = \G g_s > 0.
	\end{align*}
	Let $W = \phi g$. Using \eqref{eqofphi} and computing directly, we obtain
	\begin{align*}
		L_{rad}W = \frac{\phi}{r^2}\left(L_\rho - (N-1)\right)g.
	\end{align*}
	Therefore,
	\begin{align*}
		LW < 0 ~~ \text{ in } 0 < |x| < r_0, \quad\quad LW > 0 ~~ \text{ in } |x| > r_0.
	\end{align*}
Recall that
	\begin{align*}
		\phi = h_0r + O(r^3), \quad \phi_r = h_0 + O(r^2).
	\end{align*}
	By $k = h_0r + O(r^3)$ we have
	\begin{align*}
		g_r = r^{-1}g_s = \frac{1}{rk} = \frac{1}{h_0r^2} + O(1), \quad g = -\frac{1}{h_0r} + C + O(r).
	\end{align*}
	Then one can see
	\begin{align*}
		W_r = \phi_r g + \phi g_r = O(1), \quad W = -1 + O(r).
	\end{align*}
	So $W$ belongs to $H_{loc}^1$. As $r \to \infty$, $g = O(\ln r)$ and $\phi$ decays exponentially, implying that $W$ also decays exponentially.
	
	For
	\begin{align*}
		f \in \mathcal{H} := \big\{f \in H_{rad}^1(\R^N): f(r_0) = 0\big\},
	\end{align*}
	using local Picone identity avoiding $0$ and $r_0$ for smooth radial functions having compact support first and then taking $H^1$-approximation, we have
	\begin{align*}
		Q_L(f) = & \int_{\{|x| < r_0\}}\left(|\nabla f|^2 + \eta f^2 - (1-\theta)pv^{p-1}f^2 - \theta q v^{q-1}f^2\right)dx \\
		& + \int_{\{|x| > r_0\}}\left(|\nabla f|^2 + \eta f^2 - (1-\theta)pv^{p-1}f^2 - \theta q v^{q-1}f^2\right)dx \\
		= & \int_{\{|x| < r_0\}}W^2\left|\nabla \left(\frac{f}{W}\right)\right|^2dx + \int_{\{|x| < r_0\}}\frac{LW}{W}f^2dx \\
		& + \int_{\{|x| > r_0\}}W^2\left|\nabla \left(\frac{f}{W}\right)\right|^2dx + \int_{\{|x| > r_0\}}\frac{LW}{W}f^2dx \\
		\ge 0.
	\end{align*}
	Since the codimension of $\mathcal{H}$ is $1$, we obtain $m_{rad}(v) \le 1$.

	Then, using the decomposition in terms of spherical harmonics, it is readily seen that $m(v) \le 1$. Finally, noticing
	\begin{align*}
		Q_L(v) & = \int_{\R^N}\left(|\nabla v|^2 + \eta v^2 - (1-\theta)pv^{p+1} - \theta q v^{q+1}\right)dx \\
		& = -\int_{\R^N}\left((1-\theta)(p-1)v^{p+1} + \theta (q-1) v^{q+1}\right)dx \\
		& < 0,
	\end{align*}
	we conclude that $m(v) = 1$, completing the proof.

	\section{Positive solutions having Morse index $1$ are nondegenerate} \label{secnondegenerate}

	In this section, we focus on the proof of Theorem \ref{thmnondegenerate}. Let $u$ be a positive radial solution of \eqref{doublepower} having Morse index $1$ and let $v$ be given in \eqref{scalingv}. Equivalently, we will prove $v$ is nondegenerate in $H_{rad}^1(\R^N)$ and $H^1(\R^N)$. Noticing
	\begin{align*}
		Q_{L}(v) = -\int_{\R^N}\left((1-\theta)(p-1)v^{p+1} + \theta(q-1)v^{q+1}\right)dx < 0,
	\end{align*}
	we have
	\begin{align*}
		1 \le m_{rad}(v) \le m(v) = 1,
	\end{align*}
	namely $m_{rad}(v) = m(v) = 1$.
	
	Suppose by contradiction that $v$ is degenerate in $H_{rad}^1(\R^N)$. There exists a radial function $0 \neq w \in H^1(\R^N)$ such that $Lw = 0$. Note that $\sigma_{ess}(L) = [\eta,\infty)$. So $0$ is the second radial eigenvalue and $w(r)$ has exactly one zero point $r = r_0 > 0$. We assume that
	\begin{align*}
		w(r) > 0 ~~ (0 < r < r_0),\quad w(r) < 0 ~~ (r > r_0).
	\end{align*}
	Define
	\begin{align*}
		\chi(r) := \frac{w(r)}{v(r)}.
	\end{align*}

	\begin{lemma}
		It holds that
		\begin{align}
			\chi_r(r) < 0 \quad \text{ in } r > 0.
		\end{align}
	\end{lemma}

	\begin{proof}
		Using $Lw=0$ and \eqref{eqofv} we obtain
		\begin{align*}
			\left(r^{N-1}v^2\chi_r\right)_r = -r^{N-1}vG(v)\chi,
		\end{align*}
		where
		\begin{align*}
			G(t) := (p-1)(1-\theta)t^p + (q-1)\theta t^q.
		\end{align*}
		Moreover,
		\begin{align*}
			r^{N-1}v^2\chi_r \to 0 \quad \text{ as } r \to 0 \text{ or } r \to \infty.
		\end{align*}
		When $r \le r_0$, we have
		\begin{align*}
			r^{N-1}v^2\chi_r = -\int_0^r t^{N-1}v(t)G(v(t))\chi(t)dt < 0,
		\end{align*}
		and so $\chi_r < 0$. When $r > r_0$, from
		\begin{align*}
			r^{N-1}v^2\chi_r = \int_r^\infty t^{N-1}v(t)G(v(t))\chi(t)dt < 0,
		\end{align*}
		we deduce that $\chi_r < 0$ and complete the proof.
	\end{proof}

	\subsection{The weighted spectral method} \label{subsecspectralmethod}

	In this subsection, we assume that $N, p, q$ satisfy one of the conditions in Theorem \ref{thmmain}. By Propositions \ref{propG>01}, \ref{propG>02} and \ref{propG>03}, we have proved $\G > 0$ for $r > 0$. We will use the weighted spectral method to find a contradiction.
		
		Recall
	\begin{align*}
		\phi = -v_r > 0, \quad s = \ln r, \quad \rho(s) = r^{N-2}\phi(r)^2,
	\end{align*}
	and let
	\begin{align*}
		e(s) := \frac{w(r)}{\phi(r)}.
	\end{align*}
	Using $Lw = 0$ we get that
	\begin{align*}
		-\left(\rho e_s\right)_s = (N-1)\rho e, \quad \int_{\R}\rho eds = 0.
	\end{align*}
	It is not difficult to see that $e_s < 0$ since $w$ changes its sign exactly once and $\rho e_s \to 0$ as $s \to -\infty$ or $s \to \infty$. Using
	\begin{align*}
		-e_{ss} - \frac{\rho_s}{\rho}e_s = (N-1)e,
	\end{align*}
	one gets
	\begin{align} \label{eigenfunctionzs}
		\left(-\frac{d^2}{ds^2} - \frac{\rho_s}{\rho}\frac{d}{ds} - \left(\frac{\rho_s}{\rho}\right)_s - (N-1)\right)e_s = 0.
	\end{align}
	Define the operator
	\begin{align*}
		\mathcal{A} := U^{-1}\widetilde{H}U = -\frac{d^2}{ds^2} - \frac{\rho_s}{\rho}\frac{d}{ds} - \left(\frac{\rho_s}{\rho}\right)_s.
	\end{align*}
	In the following calculations, $\mathcal A$ is used as a differential expression on finite intervals. No global $L^2(\rho\,ds)$-integrability of $1/k$ is required.
	Direct computations yield
	\begin{align} \label{testing1k>0}
		\left(\mathcal{A} - (N-1)\right)\frac{1}{k} = \G\frac{1}{k} > 0.
	\end{align}
On any bounded interval $[-S,S]$ we obtain
	\begin{align} \label{-StoS}
		\int_{-S}^S \rho\left(e_s(\mathcal{A} - (N-1))\frac{1}{k} - \frac{1}{k}(\mathcal{A} - (N-1))e_s\right)ds  = \left[\rho \left(\frac{1}{k}e_{ss} - e_s\left(\frac{1}{k}\right)_s\right)\right]_{-S}^S.
	\end{align}
	Note that
	\begin{align*}
		\rho \left(\frac{1}{k}e_{ss} - e_s\left(\frac{1}{k}\right)_s\right) = r^{N-2}v^2\left(e_s k\right)_s.
	\end{align*}
	Let
	\begin{align*}
		\varpi = \frac{w}{v} = e k.
	\end{align*}
	Direct computations yield
	\begin{align*}
		e_s k = \varpi_s + (z-1)\varpi.
	\end{align*}
	As $r \to 0$, $s = \ln r \to -\infty$. Using
	\begin{align*}
		\varpi = w(0) + O(r^2), \quad \varpi_s = O(r^2), \quad \varpi_{ss} = O(r^2), \quad z = O(r^2), \quad z_s = O(r^2),
	\end{align*}
	we get
	\begin{align*}
		\rho \left(\frac{1}{k}e_{ss} - e_s\left(\frac{1}{k}\right)_s\right) = O(r^{N}) \to 0.
	\end{align*}
	Denote
	\begin{align*}
		l(r) = \frac{\phi_r}{\phi}, \quad j(r) = w_r - lw.
	\end{align*}
	Since
	\begin{align*}
		l = \frac{k_r}{k} - k,
	\end{align*}
	for large $r$, $l, l_r$ are bounded. Then we know $j, j_r$ decay exponentially. As $r \to \infty$, $s = \ln r \to \infty$, and we have
	\begin{align*}
		\rho \left(\frac{1}{k}e_{ss} - e_s\left(\frac{1}{k}\right)_s\right) = r^{N-1}vj + r^N\left(vj_r-v_rj\right)\to 0.
	\end{align*}
	Sending $S \to \infty$ in \eqref{-StoS} and combining \eqref{eigenfunctionzs} and \eqref{testing1k>0}, we deduce that
	\begin{align*}
		\int_{-S}^S\rho e_s \G\frac{1}{k}ds \to 0.
	\end{align*}
	However, since
	\begin{align*}
		\rho > 0, \quad e_s < 0, \quad \G > 0, \quad \frac{1}{k} > 0,
	\end{align*}
	we have
	\begin{align*}
		\int_{-S}^S\rho e_s \G\frac{1}{k}ds \le \int_{-S_0}^{S_0}\rho e_s \G\frac{1}{k}ds < 0 ~~ (S \ge S_0),
	\end{align*}
	a contradiction! This shows that $v$ is nondegenerate in $H_{rad}^1(\R^N)$. Then, using the decomposition in terms of spherical harmonics, we also obtain that $v$ is nondegenerate in $H^1(\R^N)$ and complete the proof in this case.

	\subsection{The moment-Pohozaev method} \label{subM-Pmethod}

	In this subsection we always assume that $N \ge 2$, $1 < q < p < 2^* - 1$, $N(p-1) \le 2(q+1)$. We will find a contradiction by comparing inequalities obtained by using moment identities and Pohozaev identity.

	\begin{lemma}
		We have
		\begin{align} \label{momentcompare}
			\frac{\int_{\R^N}vwdx}{\int_{\R^N}v^2dx} < \frac{\int_{\R^N}v^qwdx}{\int_{\R^N}v^{q+1}dx} < \frac{\int_{\R^N}v^pwdx}{\int_{\R^N}v^{p+1}dx}.
		\end{align}
	\end{lemma}

	\begin{proof}
		Let
		\begin{align*}
			d\mu(r) := |\mathbb{S}^{N-1}|r^{N-1}dr.
		\end{align*}
		Since $v$ and $\chi$ are strictly decreasing in $r > 0$, for any $2 \le \tau < t$ we have
		\begin{align*}
			& \left(\int_{\R^N}v^t\chi dx\right)\left(\int_{\R^N}v^\tau dx\right) - \left(\int_{\R^N}v^\tau\chi dx\right)\left(\int_{\R^N}v^tdx\right) \\
			= & \frac{1}{2}\int_0^\infty\int_0^\infty v(r)^\tau v(s)^\tau\left(\chi(r) - \chi(s)\right) \left(v(r)^{t-\tau} - v(s)^{t-\tau}\right)d\mu(r)d\mu(s) > 0,
		\end{align*}
		namely
		\begin{align*}
			\frac{\int_{\R^N}v^\tau \chi dx}{\int_{\R^N}v^\tau dx} < \frac{\int_{\R^N}v^t\chi dx}{\int_{\R^N}v^t dx}.
		\end{align*}
		Taking $\tau = 2, t = q+1$ and $\tau = q+1, t = p+1$ we complete the proof.
	\end{proof}

	By $Lw = 0$ and $Lv = -(p-1)(1-\theta)v^p - (q-1)\theta v^q$, we get
	\begin{align} \label{C+D}
		(p-1)(1-\theta)\int_{\R^N}v^pwdx + (q-1)\theta\int_{\R^N}v^qwdx = 0.
	\end{align}
	Using
	\begin{align*}
		L (x\cdot\nabla v) = 2\left((1-\theta)v^p + \theta v^q - \eta v\right),
	\end{align*}
	we deduce that
	\begin{align} \label{CDM}
		(1-\theta)\int_{\R^N}v^pwdx + \theta\int_{\R^N}v^qwdx =  \eta\int_{\R^N}vwdx.
	\end{align}
	It follows from \eqref{momentcompare} and \eqref{C+D} that
	\begin{align*}
		\int_{\R^N}v^pwdx > 0, \quad \int_{\R^N}v^qwdx < 0.
	\end{align*}
	Together with \eqref{CDM} we get
	\begin{align*}
		\eta\int_{\R^N}vwdx = \frac{p-q}{p-1}\theta\int_{\R^N}v^qwdx < 0.
	\end{align*}
	Using \eqref{momentcompare} again we conclude that
	\begin{align} \label{moment>}
		(p-q)\theta\int_{\R^N}v^{q+1}dx > (p-1)\eta\int_{\R^N}v^2dx.
	\end{align}
On the other hand, by the Nehari identity
	\begin{align*}
		\int_{\R^N}|\nabla v|^2dx + \eta\int_{\R^N}v^2dx = (1-\theta)\int_{\R^N}v^{p+1}dx + \theta\int_{\R^N}v^{q+1}dx
	\end{align*}
	and the Pohozaev identity
	\begin{align*}
		\frac{N-2}{2}\int_{\R^N}|\nabla v|^2dx + \frac{N\eta}{2}\int_{\R^N}v^2dx = \frac{N}{p+1}(1-\theta)\int_{\R^N}v^{p+1}dx + \frac{N}{q+1}\theta\int_{\R^N}v^{q+1}dx,
	\end{align*}
	one can see that
	\begin{align*}
		\eta\int_{\R^N}v^2dx = (1-\theta)\left(\frac{N}{p+1}-\frac{N-2}{2}\right)\int_{\R^N}v^{p+1}dx + \theta\left(\frac{N}{q+1}-\frac{N-2}{2}\right)\int_{\R^N}v^{q+1}dx.
	\end{align*}
	Thus
	\begin{align*}
		(p-1)\eta\int_{\R^N}v^2dx - (p-q)\theta\int_{\R^N}v^{q+1}dx = &~ (1-\theta)(p-1) \left(\frac{N}{p+1}-\frac{N-2}{2}\right)\int_{\R^N}v^{p+1}dx \\
		& + \theta(q-1)\left(1 - \frac{N(p-1)}{2(q+1)}\right)\int_{\R^N}v^{q+1}dx.
	\end{align*}
	Since $p + 1 < 2^*$ and $N(p-1) \le 2(q+1)$, we find a contradiction with \eqref{moment>}. This proves $v$ is nondegenerate in $H_{rad}^1(\R^N)$. Then, using the decomposition in terms of spherical harmonics, we also obtain that $v$ is nondegenerate in $H^1(\R^N)$ and complete the proof in this case.	

	\section{The uniqueness of positive solutions: the fixed frequency case} \label{secuniqueness1}

	This section is devoted to proving Theorem \ref{thmmain}. In view of Theorem \ref{thmMorse}, it is enough to obtain the following result.

	\begin{proposition} \label{propweakunique}
		Under the assumptions of Theorem \ref{thmnondegenerate}, radial solutions of \eqref{doublepower} having Morse index $1$ are unique.
	\end{proposition}

	\begin{proof}
		We define the action functional $J_\la : H^1(\R^N) \to \R$ as
		\begin{align*}
			J_\la(u) = \frac{1}{2}\int_{\R^N}\left(|\nabla u|^2 + \la u^2\right)dx - \frac{1}{p+1}\int_{\R^N}|u|^{p+1}dx - \frac{1}{q+1}\int_{\R^N}|u|^{q+1}dx.
		\end{align*}
		The functional $J_\la$ satisfies the Palais-Smale condition in the radial space $H_{rad}^1(\R^N)$. Any critical point of $J_\la$ constrained on $H_{rad}^1(\R^N)$ is a weak solution of
		\begin{align} \label{eqofweaksolution}
			-\Delta u + \la u = |u|^{p-1}u + |u|^{q-1}u \quad \text{ in } \R^N.
		\end{align}
		If the critical point is positive, it is indeed a solution of \eqref{doublepower}. Define the radial Nehari manifold
		\begin{align*}
			\mathcal{N}_\la := \big\{u \in H_{rad}^1(\R^N) \backslash \{0\}: \int_{\R^N}\left(|\nabla u|^2 + \la u^2\right)dx = \int_{\R^N}(|u|^{p+1} + |u|^{q+1})dx \big\},
		\end{align*}
		and set
		\begin{align*}
			c_\la := \inf_{\mathcal{N}_\la}J_\la.
		\end{align*}
		Then we can find a non-negative critical point $u_\la \in H_{rad}^1(\R^N)$ of $J_\la$ at level $c_\la$, namely
		\begin{align*}
			J_\la(u_\la) = c_\la.
		\end{align*}
		Using the strong maximum principle, one sees that $u_\la$ is positive. Since the codimension of $\mathcal{N}_\la$ in $H_{rad}^1(\R^N)$ is $1$, we have $m_{rad}(u_\la) \le 1$. Together with
		\begin{align*}
			\langle L_{u_\la}u_\la,u_\la \rangle =(1-p)\int_{\R^N}|u_\la|^{p+1}dx + (1-q)\int_{\R^N}|u_\la|^{q+1}dx < 0,
		\end{align*}
		we derive that $m_{rad}(u_\la) = 1$. Then, using the decomposition in terms of spherical harmonics, we also obtain that $m(u_\la) = 1$.

		For any $\la_* > 0$, Theorem \ref{thmnondegenerate} yields that $u_{\la_*}$ is nondegenerate in $H_{rad}^1(\R^N)$. Applying the implicit function theorem to the map
		\begin{align*}
			H^1_{rad}(\R^N) \times (0,\infty) \to \left(H^1_{rad}(\R^N)\right)^*: ~~ (u,\la) \mapsto -\Delta u + \la u - |u|^{p-1}u - |u|^{q-1}u,
		\end{align*}
		we can establish a local branch parameterized by $\la$
		\begin{align*}
			(\la_*-\ep,\la_*+\ep) \to H_{rad}^1(\R^N), \quad \la \mapsto u(\la),
		\end{align*}
		such that $\ep > 0$ is small, $u(\la_*) = u_{\la_*}$, the map $\la \mapsto u(\la)$ is $C^1\left((\la_*-\ep,\la_*+\ep), H_{rad}^1(\R^N)\right)$, $u(\la)$ solves \eqref{eqofweaksolution} and $m(u(\la)) = 1$. Define the operator
		\begin{align*}
			L_{\la,-} := -\Delta + \la - |u(\la)|^{p-1} - |u(\la)|^{q-1}.
		\end{align*}
		Note that $L_{\la,-}u(\la) = 0$ and $\sigma_{ess}(L_{\la,-}) = [\la,\infty)$. Since $u(\la_*)$ is positive, $0$ is the first eigenvalue of $L_{\la_*,-}$ and is simple. It is clear that $L_{\la,-} \to L_{\la_*,-}$ in norm-resolvent sense. Taking $\ep > 0$ small enough we can get that $0$ is the first eigenvalue of $L_{\la,-}$ and $u(\la)$ is positive for $\la \in (\la_*-\ep,\la_*+\ep)$. Next, we extend the branch to $\la \to \infty$. For this, we define
		\begin{align*}
			\la^* := \sup \big\{\Lambda: \text{ for } \la \in [\la_*,\Lambda), & ~ u(\la) \text{ is a positive solution of } \eqref{eqofweaksolution}, m(u(\la)) = 1,\\
			& ~~~ u(\la_*) = u_{\la_*}, \la \mapsto u(\la) \text{ is } C^1\left([\la_*,\Lambda), H_{rad}^1(\R^N)\right) \big\}
		\end{align*}
		We claim $\la^* = \infty$. Suppose by contradiction that $\la^* < \infty$. Next, we prove that $u(\la)$ is uniformly bounded for $\la \in [\la_*,\la^*)$. Define
		\begin{align*}
			g(\la) := \frac{p-1}{p+1}\int_{\R^N}|u(\la)|^{p+1}dx + \frac{q-1}{q+1}\int_{\R^N}|u(\la)|^{q+1}dx.
		\end{align*}
		By the Nehari identity we have $C_1g(\la) \le \|u(\la)\|_{H^1(\R^N)}^2 \le C_2g(\la)$ and
		\begin{align*}
			\la \int_{\R^N}u(\la)^2dx \le \frac{q+1}{q-1} g(\la).
 		\end{align*}
		Let $v_\la = \frac{d}{d \la}u(\la)$ and we have
		\begin{align*}
			-\Delta v_\la + \la v_\la + u(\la) = p|u(\la)|^{p-1}v_\la + q|u(\la)|^{q-1}v_\la.
		\end{align*}
		Then we get
		\begin{align*}
			g'(\la) = \int_{\R^N}u(\la)^2 dx \le \frac{q+1}{q-1}\la^{-1} g(\la),
		\end{align*}
		yielding
		\begin{align*}
			g(\la) \le \left(\frac{\la}{\la_*}\right)^{\frac{q+1}{q-1}}g(\la_*) \le \left(\frac{\la^*}{\la_*}\right)^{\frac{q+1}{q-1}}g(\la_*).
		\end{align*}
		Namely, there exists $C$ depending on $\la_*$ and $\la^*$ but independent of $\la$ such that
		\begin{align*}
			\|u(\la)\|_{H^1(\R^N)} \le C.
		\end{align*}
		One can check that $c_\la > 0$ is continuous and increasing in $\la > 0$, see for example \cite{Song}. Take $\{\la_n\} \subset [\la_*,\la^*)$ such that $\la_n \to \la^*$. The sequence $\{u(\la_n)\}$ is a bounded Palais-Smale sequence of $J_{\la^*}$. By the Palais-Smale condition, up to a subsequence, there exists $\tilde{u} \in H_{rad}^1(\R^N)$ such that $u(\la_n) \to \tilde{u}$ strongly in $H^1(\R^N)$. Clearly, $\tilde{u}$ is non-negative. From
		\begin{align*}
			J_{\la^*}(\tilde{u}) = \lim_{n}J_{\la_n}(u(\la_n)) \ge \lim_{n}c_{\la_n} = c_{\la^*} > 0,
 		\end{align*}
		we deduce that $\tilde{u} \neq 0$. By the strong maximum principle, $\tilde{u}$ is positive. Combining
		\begin{align*}
			m(\tilde{u}) \le \liminf_{n}m(u(\la_n)), \quad \langle L_{\tilde{u}}\tilde{u},\tilde{u} \rangle < 0,
		\end{align*}
		and the decomposition in terms of spherical harmonics, we can check that $m(\tilde{u}) = 1$. By Theorem \ref{thmnondegenerate}, $\tilde{u}$ is nondegenerate. Applying the implicit function theorem at $(\tilde{u}, \la^*)$ with $u(\la^*) = \tilde{u}$, we find a contradiction with the definition of $\la^*$.

		Denote $u_\la = u(\la)$ for $\la \in [\la_*,\infty)$. We consider
		\begin{align*}
			w_\la(x) := \la^{-\frac{1}{p-1}}u_\la\left(\frac{x}{\sqrt{\la}}\right),
		\end{align*}
		which satisfies the following equation
		\begin{align} \label{scalingequation}
			-\Delta w + w = w^p + \la^{\frac{q-p}{p-1}}w^q \quad \text{ in } \R^N, \quad w > 0.
		\end{align}
		Let $w_0$ be the unique radial positive solution of
		\begin{align*}
			-\Delta w + w = w^p, \quad w \in H^1(\R^N).
		\end{align*}
		One can see that $w_\la \to w_0$ in $H^1(\R^N)$ as $\la \to \infty$, see \cite[Theorem 4.10]{JZZ}.
		
		If at $\la =\la_*$, \eqref{doublepower} admits another radial solution $\widehat{u}_{\la_*} \neq u_{\la_*}$ with $m(\widehat{u}_{\la_*}) = 1$, then we obtain another branch
		\begin{align*}
			[\la_*,\infty) \to H^1_{rad}(\R^N), \quad \la \mapsto \widehat{u}(\la)
		\end{align*}
		such that $\widehat{u}(\la_*)= \widehat{u}_{\la_*}$, the map $\la \mapsto \widehat{u}(\la)$ is $C^1$, $\widehat{u}(\la)$ is a positive solution of \eqref{eqofweaksolution} and $m(\widehat{u}(\la)) = 1$. Denote $\widehat{u}_\la = \widehat{u}(\la)$ and we consider
		\begin{align*}
			\widehat{w}_\la(x) := \la^{-\frac{1}{p-1}}\widehat{u}_\la\left(\frac{x}{\sqrt{\la}}\right),
		\end{align*}
		which is also a solution of \eqref{scalingequation}. As $\la \to \infty$, we have $\widehat{w}_\la \to w_0$ in $H^1(\R^N)$. Since $w_\la \neq \widehat{w}_\la$, using the nondegeneracy of $w_0$ and applying the implicit function theorem, we find a contradiction and complete the proof.
	\end{proof}

	\begin{proof}[Proof of Theorem \ref{thmmain}]
		Under the assumptions of Theorem \ref{thmmain}, conditions in both Theorem \ref{thmMorse} and Theorem \ref{thmnondegenerate} are satisfied. Combining results in Theorem \ref{thmMorse} and Proposition \ref{propweakunique} we complete the proof of Theorem \ref{thmmain}.
	\end{proof}

	\section{The uniqueness of positive solutions: the fixed mass case} \label{secuniqueness2}
	
	In this section, we prove Theorem \ref{thmmain2}. For this purpose, we first give some lemmas.
	
	\begin{lemma} \label{atmosttwozero}
		Let $\sigma$ be a finite signed measure on $[0,\infty)$ satisfying $\int (1+y^2)d|\sigma|(y) < \infty$. Assume that for some $0 < s_1 < s_2 < \infty$, its signs on $(0,s_1)$, $(s_1,s_2)$ and $(s_2,\infty)$ are $(-,+,-)$ with nonzero mass in every interval and no atoms at $0$, $s_1$, $s_2$. Then
		\begin{align*}
			\varphi(t) = \int_0^\infty e^{-ty}d\sigma(y), \quad t \ge 0,
		\end{align*}
		has at most two distinct zeros on $[0,\infty)$.
	\end{lemma}
	
	\begin{proof}
		Suppose by contradiction that $0 \le t_0 < t_1 < t_2$ are three distinct zeros of $\varphi(t)$. Applying Rolle's theorem for $e^{s_1t}\varphi$ gives $c_i \in (t_{i-1},t_i)$, $i = 1,2$ such that $T(c_i) = 0$ where
		\begin{align*}
			T(t) = \int_0^\infty (s_1-y)e^{-ty}d\sigma(y).
		\end{align*}
		Set $K(t) = e^{s_2t}T(t)$. Then
		\begin{align*}
			K'(t) = e^{s_2t} \int_0^\infty (s_2-y)(s_1-y)e^{-ty}d\sigma(y) < 0,
		\end{align*}
		contradicting $K(c_1) = K(c_2) = 0$.
	\end{proof}
	
	\begin{lemma} \label{asymptoticlemma}
		Let $u$ be a radial solution of \eqref{doublepower} and let $v \in H_{rad}^1(\R^N)$ satisfy
		\begin{align} \label{Luv=-u}
			L_u v = -u.
		\end{align}
		Then there exists $C > 0$ such that as $r \to \infty$,
		\begin{align}
			& u(r) = C r^{-\frac{N-1}{2}}e^{-\sqrt{\la}r}\left(1 + o(1)\right), \label{expandingu}\\
			& v(r) = -\frac{C}{2\sqrt{\la}}r^{-\frac{N-3}{2}}e^{-\sqrt{\la}r}\left(1 + o(1)\right). \label{expandingv}
		\end{align}
		In particular, $v(r) < 0$ for sufficiently large $r$.
	\end{lemma}

	\begin{proof}
		The expansion \eqref{expandingu} is the standard radial bound-state asymptotic. Next we prove \eqref{expandingv}. Set $m = (N-1)/2$, $U = r^m u$, $V = r^m v$ and $c_N = (N-1)(N-3)/4$. Then \eqref{Luv=-u} becomes
		\begin{align*}
			-V'' + \left(\la + \eta(r)\right)V = -U, \eta(r) = \frac{c_N}{r^2} - pu^{p-1} - qu^{q-1} = O(r^{-2}).
		\end{align*}
		We solve that
		\begin{align*}
			V(r) = C_1e^{-\sqrt{\la}r} + \frac{1}{2\sqrt{\la}}\left(e^{-\sqrt{\la}r}\int_R^r e^{\sqrt{\la}s}f(s)ds + e^{\sqrt{\la}r}\int_r^\infty e^{-\sqrt{\la}s}f(s)ds\right), \quad f = -U - \eta V.
		\end{align*}
		Taking $R$ large and closing this Volterra estimate first yields $V = O\left((1+r)e^{-\sqrt{\la}r}\right)$. Since $U = Ce^{-\sqrt{\la}r}\left(1 + o(1)\right)$,
		\begin{align*}
			\int_R^r e^{\sqrt{\la}s}(-U(s))ds = -Cr + o(r), \quad \int_R^r e^{\sqrt{\la}s}\eta(s)V(s)ds = O(\ln r) = o(r).
		\end{align*}
		Hence,
		\begin{align*}
			V= -\frac{C}{2\sqrt{\la}}re^{-\sqrt{\la}r} + o(re^{-\sqrt{\la}r}),
		\end{align*}
		proving \eqref{expandingv}.
	\end{proof}

	\begin{lemma} \label{intuvneq0}
		Let $u$ and $v$ be given in Lemma \ref{asymptoticlemma}. Assume that
		\begin{align*}
			N\ge 2, \quad 1 < q < p \le 1 + \frac{4}{N}.
		\end{align*}
		Then
		\begin{align*}
			\int_{\R^N}uv dx \neq 0.
		\end{align*}
	\end{lemma}

	\begin{proof}
		Arguing by contradiction, we assume
		\begin{align*}
			\int_{\R^N}uv dx = 0.
		\end{align*}
		For $a \ge 1$, we define
		\begin{align*}
			H(a) := \int_{\R^N}u^avdx, \quad M = \int_{\R^N}u^2dx.
		\end{align*}
		By \eqref{Luv=-u} and
		\begin{align*}
			L_u u = (1-p)u^p + (1-q)u^q,
		\end{align*}
		we derive that
		\begin{align} \label{M}
			(p-1)H(p) + (q-1)H(q) = M.
		\end{align}
		Furthermore, using
		\begin{align*}
			L_u\left(x\cdot \nabla u\right) = 2(-\Delta u) = 2 \left(u^p + u^q - \la u\right)
		\end{align*}
		and
		\begin{align*}
			\int_{\R^N}u \left(x \cdot \nabla u\right)dx = -\frac{N}{2}M,
		\end{align*}
		we obtain that
		\begin{align} \label{N4M}
			H(p) + H(q) = \frac{N}{4}M.
		\end{align}
		Combining \eqref{M} and \eqref{N4M} we solve that
		\begin{align}
			& (p-q)H(q) = \left(\frac{N}{4}(p-1) - 1\right)M \le 0, \label{eqp-qHq} \\
			& (p-q)H(p) = \left(1 - \frac{N}{4}(q-1)\right)M > 0. \label{eqp-qHp}
		\end{align}

		Next we control the number of zeros of $v$ in $r > 0$. By Theorem \ref{thmMorse}, we have $m(u) = 1$. On the one hand, since $u$ is positive, $\int_{\R^N}uv = 0$ yields that the set $\{r > 0: v(r) > 0\}$ is not empty. On the other hand, the set $\{r > 0: v(r) > 0\}$ must be a connected interval. Otherwise, two positive components, together with \eqref{Luv=-u}, would give a two-dimensional negative subspace, in conflict with $m(u) = 1$. Hence, there exists $0 \le r_1 < r_2$ such that
		\begin{align*}
			v(r) > 0 ~~ (r_1 < r < r_2), \quad v(r) \le 0 ~~ (r \in (0,\infty) \backslash (r_1,r_2)).
		\end{align*}
		According to Lemma \ref{asymptoticlemma}, we know $r_2 < \infty$. We claim that $v$ has no zero point in $(0,\infty)\backslash[r_1,r_2]$. Suppose by contradiction that there exists $0 < r^* \notin [r_1,r_2]$ such that $v(r^*) = 0$. It is clear that $r^*$ is a local maximum point and so $v'(r^*) = 0$, $v''(r^*) \le 0$. However, \eqref{Luv=-u} yields
		\begin{align*}
			-v''(r^*) = -u(r^*) < 0,
		\end{align*}
		a contradiction!

		If $r_1 = 0$, we have
		\begin{align*}
			H(q) = \int_{\R^N}\left(u^{q-1} - u(r_2)^{q-1}\right)uvdx > 0,
		\end{align*}
		contradicting \eqref{eqp-qHq}.

		If $r_1 > 0$, we set
		\begin{align*}
			y(r) = \ln \frac{u(0)}{u(r)} > 0, \quad s_i = y(r_i), ~ i = 1,2,
		\end{align*}
		and define
		\begin{align*}
			d\sigma(y) = u(r)v(r)|\mathbb{S}^{N-1}|r^{N-1}dr.
		\end{align*}
		The measure $\sigma$ has signs $(-,+,-)$ on the intervals cut out by $s_1$ and $s_2$, and
		\begin{align*}
			\varphi(t) = \int_0^\infty e^{-ty}d\sigma(y) = u(0)^{-t}H(t+1), \quad t \ge 0.
		\end{align*}
		By Lemma \ref{asymptoticlemma},
		\begin{align*}
			\int_0^\infty \left(1 + y^2\right)d|\sigma|(y) < \infty.
		\end{align*}
		Applying Lemma \ref{atmosttwozero}, the function $\varphi(t)$ has at most two distinct zeros on $[0,\infty)$. However,
		\begin{align*}
			\varphi(0) = H(1) = \int_{\R^N}uv dx = 0,
		\end{align*}
		\eqref{eqp-qHq} and \eqref{eqp-qHp} produce a zero in $[q-1,p-1)$. Take $0 < c < s_1$ and let
		\begin{align*}
			k = -\sigma\left([0,c]\right) > 0, \quad l = \sigma\left((s_1,s_2)\right) > 0.
		\end{align*}
		From
		\begin{align*}
			\varphi(t) \le -ke^{-tc} + le^{-ts_1},
		\end{align*}
		we derive that $\varphi(t) < 0$ for sufficiently large $t$. Combining \eqref{eqp-qHp} we find a third zero point, a contradiction!
	\end{proof}

	Finally, we give the proof of Theorem \ref{thmmain2}.
	
	\begin{proof}[Proof of Theorem \ref{thmmain2}]
		By the Pohozaev identity, \eqref{doublepower} has no solution when $\la \le 0$. Note that conditions in Theorem \ref{thmmain} are satisfied. Hence, for any $\la > 0$, \eqref{doublepower} admits a unique solution $u_\la$ in $H^1_{rad}(\R^N)$. Moreover, as shown in the proof of Theorem \ref{thmmain}, these solutions form a global $C^1$ branch parameterized by $\la$, see also \cite{Song} for a proof. Then, to prove Theorem \ref{thmmain2}, a crucial step is to show that $\int_{\R^N}u_\la^2dx$, viewed as a function of $\la$, is monotonous. Let
		\begin{align*}
			v_\la := \frac{d}{d\la}u_\la.
		\end{align*}
		We have $L_{u_\la}v_\la = - u_\la$. By Lemma \ref{intuvneq0},
		\begin{align} \label{ulavla}
			\frac{d}{d\la}\int_{\R^N}u_\la^2dx = 2\int_{\R^N} u_\la v_\la dx \neq 0.
		\end{align}
		As in the proof of Theorem \ref{thmmain},
		\begin{align*}
			\la^{-\frac{1}{p-1}}u_\la\left(\frac{x}{\sqrt{\la}}\right) \to w_0 \quad \text{ in } H^1(\R^N) \quad \text{ as } \la \to \infty,
		\end{align*}
		where $w_0$ is the unique radial positive solution of
		\begin{align*}
			-\Delta w + w = w^p, \quad w \in H^1(\R^N).
		\end{align*}
		Hence,
		\begin{align*}
			\int_{\R^N}u_\la^2dx \sim \la^{\frac{4-N(p-1)}{2(p-1)}}\int_{\R^N}w_0^2dx \quad \text{ as } \la \to \infty.
		\end{align*}
		If $p <  1 + 4/N$,
		\begin{align*}
			\int_{\R^N}u_\la^2dx \to \infty \text{ as } \la \to \infty.
		\end{align*}
		If $p = 1 + 4/N$,
		\begin{align*}
			\lim_{\la \to \infty}\int_{\R^N}u_\la^2dx = c_0,
		\end{align*}
		where $c_0$ is given in Section \ref{secintroduction}.  Similarly (see
		Theorem 2.5 and Theorem 4.6 in \cite{JZZ}),
		\begin{align*}
			\la^{-\frac{1}{q-1}}u_\la\left(\frac{x}{\sqrt{\la}}\right) \to W_q \quad \text{ in } H^1(\R^N) \quad \text{ as } \la \to 0,
		\end{align*}
		where $W_q$ is the unique radial positive solution of
		\begin{align*}
			-\Delta W + W = W^q, \quad W \in H^1(\R^N).
		\end{align*}
		Hence,
		\begin{align*}
			\int_{\R^N}u_\la^2dx \sim \la^{\frac{4-N(q-1)}{2(q-1)}} \int_{\R^N}W_q^2dx \to 0 \quad \text{ as } \la \to 0.
		\end{align*}
		Recalling \eqref{ulavla}, we further have
		\begin{align*}
			\frac{d}{d\la}\int_{\R^N}u_\la^2dx > 0
		\end{align*}
		and conclude the results presented in Theorem \ref{thmmain2}.
	\end{proof}

\end{document}